\documentclass[reqno]{amsart}
\usepackage{amsmath,amsfonts,amssymb,amsthm}
\usepackage[usenames,dvipsnames]{xcolor}
\usepackage[colorlinks=true]{hyperref} 
\usepackage{url}
\hypersetup{linkcolor=Violet,citecolor=PineGreen}
\newtheorem{teor}{Theorem}[section]
\newtheorem{lema}[teor]{Lemma}
\newtheorem{prop}[teor]{Proposition}
\newtheorem{coro}[teor]{Corollary}
\theoremstyle{definition}
\newtheorem{defi}[teor]{Definition}

\newtheorem{hipos}[teor]{Hypotheses}
\newtheorem{nota}[teor]{Remark}

\numberwithin{equation}{section}

\newcommand{\N}{\mathbb{N}}

\newcommand{\R}{\mathbb{R}}

\newcommand{\mB}{\mathcal{B}}
\newcommand{\mC}{\mathcal{C}}
\newcommand{\mD}{\mathcal{D}}

\newcommand{\mK}{\mathcal{K}}

\newcommand{\mM}{\mathcal{M}}
\newcommand{\mO}{\mathcal{O}}
\newcommand{\mP}{\mathcal{P}}

\newcommand{\mU}{\mathcal{U}}
\newcommand{\mV}{\mathcal{V}}

\newcommand{\ep}{\varepsilon}
\newcommand{\pu }{{\cdot}}

\newcommand{\WW}{W^{1,\infty}}
\newcommand{\W}{\Omega}
\newcommand{\w}{\omega}
\newcommand{\wt}{\w\pu t}
\newcommand{\ws}{\w\pu s}

\newcommand{\wit}{\widetilde}

\newcommand{\n}[1]{\|#1\|}
\newcommand{\Frac}[2]{\displaystyle\frac{#1}{#2}}

\newcommand{\Sum}{\displaystyle\sum }

\newcommand{\lsm}{\left[\begin{smallmatrix}}
\newcommand{\rsm}{\end{smallmatrix}\right]}

\newcommand{\Lin}{\text{\rm Lin}}

\begin{document}
\title[Nonautonomous FDEs with state-dependent delay]
{Exponential stability for nonautonomous functional differential
equations with state-dependent delay}
\author[I.~Maroto]{Ismael Maroto}
\author[C.~N\'{u}\~{n}ez]{Carmen N\'{u}\~{n}ez}
\author[R.~Obaya]{Rafael Obaya}
\address{Departamento de Matem\'{a}tica Aplicada, Universidad de
Valladolid, Paseo del Cauce 59, 47011 Valladolid, Spain}
\email[Ismael Maroto]{ismmar@eii.uva.es}
\email[Carmen N\'{u}\~{n}ez]{carnun@wmatem.eis.uva.es}
\email[Rafael Obaya]{rafoba@wmatem.eis.uva.es}
\thanks{Partly supported by MEC (Spain)
under project MTM2015-66330-P and by European Commission under project
H2020-MSCA-ITN-2014.}
\date{}
\begin{abstract}
The properties of stability of a compact semiflow $(\mK,\Pi,\R^+)$ determined by a
family of nonautonomous FDEs with state-dependent delay taking values in $[0,r]$
are analyzed. The solutions of the variational equation through
the orbits of $\mK$ induce linear skew-product semiflows on the bundles $\mK\times
\WW([-r,0],\R^n)$ and $\mK\times C([-r,0],\R^n)$. The coincidence
of the upper-Lyapunov exponents for both
semiflows is checked, and it is a fundamental tool to prove that the
strictly negative character of this upper-Lyapunov exponent is equivalent to the
exponential stability of $\mK$ in $\W\times\WW([-r,0],\R^n)$
and also to the exponential stability of this minimal set when the
supremum norm is taken in $\WW([-r,0],\R^n)$.
In particular, the existence of a uniformly exponentially stable solution
of a uniformly almost periodic FDE ensures the existence of
exponentially stable almost periodic solutions.
\end{abstract}
\maketitle
\section{Introduction}\label{sec1}
State-dependent delay differential equations (SDDEs for short)
have been extensively investigated during the last years,
due to the theoretical interest of the related problems and to the great number of
potential applications in many areas of interest, as
automatic control, mechanical engineering, neural networks,
population dynamics and ecology. Among the
extensive list of works devoted to this field, we can mention
Hartung~\cite{hart1,hart2,hart3,hart4}, Wu~\cite{wu}, Walther \cite{walt2,walt},
Hartung {\em et al.}~\cite{hkvv}, Chen {\em et al.} \cite{chhw},
Hu and Wu \cite{huwu}, Mallet-Paret and Nussbaum \cite{mpna}, Hu {\em et al.} \cite{huvz},
Barbarossa and Walther \cite{bawa},
and He and de la Llave \cite{hell,hell2}, and Krisztin and  and Rezounenko
\cite{krre},
as well as the many references therein.
\par
In this paper, we analyze the exponential stability properties of the solutions of a
nonautonomous SDDE.
The use of the skew-product formulation allows us to use techniques arising from the
topological dynamics.
\par
More precisely, let $(\W,\sigma,\R)$ be a continuous flow on a compact metric space.
We write $\wt:=\sigma(t,\w)$, and consider
the family of SDDEs with maximum delay $r>0$, given by
\begin{equation}\label{1.introF}
 \dot{y}(t)=F(\wt,y(t), y(t-\tau(\wt,y_{t})))\,,\qquad t>0
\end{equation}
for $\w\in\W$,
where $F\colon\W\times\R^{n}\times\R^{n}\to\R^{n}$ is continuous and admits
continuous partial derivatives with respect to the vectorial components. Let
$C:=C([-r,0],\R^n)$ be endowed with the supremum norm. The state-dependent
delay is given by a continuous function $\tau\colon\W\times C\to[0,r]$, which is supposed to be
continuously differentiable with respect to its second argument and to satisfy some
standard Lipschitz conditions. And, as usual, we
represent $y_t(s):=y(t+s)$ for $s\in[-r,0]$ whenever
$y$ is a continuous function on $[t-r,t]$.
\par
It is well known that such a family may arise from a single SDDE, namely
$\dot{y}(t)=f(t, y(t), y(t-\wit\tau(t,y_{t})))$.
Standard conditions on the temporal variation of the map
$(f,\wit\tau)\colon\R\times\R^{n}\times\R^{n}\times C\to\R^n\times[0,r]$
(which are satisfied in the uniformly almost-periodic case,
but also in much more general situations), ensure that its {\em hull\/}
$\W$ (i.e., the closure in the compact-open topology of the
set of time-translated functions $(f,\wit\tau)_s(t,x,y,v):=(f,\wit\tau)(t+s,x,y,v)$)
is a compact metric space supporting a continuous flow $\sigma$, which is also given by
time-translation: the elements of are $\W$ are functions
$\w=(\w_1,\w_2)\colon\R\times\R^{n}\times\R^{n}\times C\to\R^n\times[0,r],\;
(t,y_1,y_2,v)\mapsto(\w_1(t,y_1,y_2),\w_2(t,v))$, and the continuous
flow is given by the map $\sigma:\R\times\W\to\W\,,\;
(t,(\w_1,\w_2))\mapsto(\w_1,\w_2)\pu t$, where
$((\w_1,\w_2)\pu t)(s,y_1,y_2,v)=(\w_1(t+s,y_1,y_2),\w_2(t+s,v))$). These conditions
also ensure that $F(\w,y_1,y_2)=\w_1(0,y_1,y_2)$
and $\tau(\w,v)=\w_2(0,v)$ for $\w=(\w_1,\w_2)$
are continuous operators: see Hino {\em et al.}~\cite{himn}.
In this way be obtain a family of the type \eqref{1.introF}
which includes the initial equation: just take
$\w=(f,\wit\tau)$, and note that, in particular, it has a dense orbit in $\W$).
In addition, it turns out that any of the
equations of the family satisfies the hypotheses assumed on the initial one.
Once this formulation is established, the analysis of the dynamical behavior
of the whole family provide information on the solutions of the initial equation.
\par
Additional recurrence conditions can be assumed on $f$ and $\wit\tau$ in
order to ensure that the flow $(\W,\sigma,\R)$ is minimal.
This is the situation in the particular cases for which the pair
$(f,\wit\tau)$ is uniformly periodic, almost periodic or almost automorphic,
properties which in fact ensure the same for the flow on the corresponding hull.
However, our approach in this paper is more general:
we assume neither that~\eqref{1.introF} comes
from a single SDDE, nor the minimality of $(\W,\sigma,\R)$. This last condition will be
indeed required for some of the results, but it will be imposed in due time.
\par
The compact metric space $\W$ is the
base of the bundle which constitutes the phase space of a
skew-product semiflow, whose fiber component is determined by the
solutions of the family~\eqref{1.introF}.
In our setting, the fiber of the bundle will be the Banach space $\WW\subset C$
of the Lipschitz-continuous functions endowed with the standard norm.
The already mentioned conditions assumed on the vector field and on the delay are intended to
ensure the existence, uniqueness and some regularity properties of the solutions. It is
convenient to keep in mind the idea that they are more exigent than those ensuring
similar properties in the study of fixed or time-dependent delay equations.
Strongly based on previous results
of \cite{hart1}, we have established in \cite{mano1}
the existence of a unique maximal solution
$y(t,\w,x)$ of the equation \eqref{1.introF} given by $\w\in\W$ for every initial
data $x\in\WW$ (i.e., with $y(s,\w, x)=x(s)$ for $s\in[-r,0]$),
which are defined on $[-r, \beta_{\w, x})$, with $\beta_{\w,x}\le\infty$. Since,
if $t\in[0,\beta_{\w,x})$, the map $u(t,\w,x)(s):=y(t+s,\w,x)$ belongs
to $\WW$, then
\eqref{1.introF} determines the local skew-product semiflow on $\W\times\WW$
\[
 \Pi\colon\mathcal{U}\subseteq\R^+\!\times\W\times\WW\to\W\times\WW,
 \quad (t,\w, x)\mapsto(\wt,u(t,\w, x))\,.
\]
In general, this semiflow is not continuous. But it satisfies strong
continuity properties, described in Theorem~\ref{3.flujocontinuo} below.
We will call it a {\em pseudo-continuous\/} semiflow.
Its interest relies on
the fact that, despite the lack of global continuity, it allows us to
use the classical tools of topological dynamics in the analysis of the
behaviour of its orbits, i.e., in the qualitative analysis of the solutions of
\eqref{1.introF}. In particular,
the restriction of $\Pi$ to positively invariant compact subsets is continuous.
\par
Let $\mK\subset\W\times\WW$ be a positively $\Pi$-invariant subset
projecting over the whole base and containing backward extensions of all its points.
One of our main goals in this paper is to characterize the exponential stability of $\mK$
and of the semiorbits that it contains in terms of the Lyapunov exponents
of its elements with respect to the linearized semiflow. More precisely, let us
introduce the set of pairs \lq\lq(equation, initial data)" which satisfy
the compatibility condition given by the vector field, namely
\[
 \mC_0:=\{(\w,x)\in\W\times C^1\,|\;\dot{x}(0^-)=F(\w,x(0), x(-\tau(\w,x)))\}\,,
\]
and define $L\colon\mC_0\to\mathcal{L}(\WW, \R^{n})$ by
\[
\begin{split}
L(\w,x)\phi&:=D_2F(\w,x(0), x(-\tau(\w,x)))\phi(0)\!+
\!D_3F(\w,x(0),x(-\tau(\w,x)))\phi(-\tau(\w,x))\\
&\quad\,-D_3F(\w,x(0), x(-\tau(\w,x)))\dot{x}(-\tau(\w,x))\pu D_2\tau(\w,x)\phi\,.
\end{split}
\]
The results of Section 3 of \cite{hart1} and Section 4 of~\cite{mano1} prove that,
if $(\w,x)\in\mC_0$ and $t\in[0,\beta_{\w,x})$,
then: there exists the linear map $u_x(t,\w,x)\colon\WW\to\WW$ and is continuous;
it determines the Fr\'{e}chet derivative of $u(t,\w,x)$ with respect to $x$;
and $(u_{x}(t,\w,x)\,v)(s)=z(t+s,\w,x,v)$ where $z(t,\w,x,v)$
is the solution of the variational equation $\dot{z}(t)=L(\Pi(t,\w,x))\,z_t$ with
$z(s)=v(s)$ for $s\in[-r, 0]$.
In addition, $\mK\subset\mC_0$, which allows us to
consider the linear skew-product semiflow
\begin{equation}\label{1.PiL}
\Pi_L\colon\R^+\!\times\mK\times\WW\to\mK\times\WW,\quad
(t,\w,x,v)\mapsto(\Pi(t,\w,x),w(t,\w,x,v))
\end{equation}
for $w(t,\w,x,v)(s)=z(t+s,\w,x,v)$ (which is a new pseudo-continuous semiflow)
in order to define the Lyapunov exponents and to derive the stability properties
of $\mK$ from the characteristics of these exponents. Note that this question is not
trivial, since: $\mC_0$ has empty interior, and there are
cases for which the map $u_{x}(t,\w,x)$ is not defined for all
$(\w,x)\in \W\times\WW$, i.e., for which $u(t,\w,x)$ does not
admit directional derivatives in $\WW$ (see \cite{hart4}).
\par
Let us briefly explain the structure and main results of the paper.
In Section~\ref{sec2}, we introduce the concepts of topological
dynamics required in the following pages. We also recall the definition
of exponential stability, and the notion and basic properties of the
upper Lyapunov exponent for a positively invariant compact set
(in the terms of Sacker and Sell~\cite{SackerSellDich},
Chow and Leiva \cite{chle1,chle2}, and Shen and Yi \cite{shyi}).
\par
In Section~\ref{sec3} we describe in detail the family of SDDEs and analyze some
of its properties. In particular, we prove that every bounded and
positively $\Pi$-invariant
set contains a positively $\Pi$-invariant compact subset which is maximal for
the property of existence of backward extension of its semiorbits.
In the rest of the Introduction, $\mK$ will be a
positively $\Pi$-invariant compact subset such that all its elements
admit backward extension in it. Such a set $\mK$ is contained in $\mC_0$,
and so we can define
the semiflow $\Pi_L$ on $\mK\times\WW$ by \eqref{1.PiL}.
In addition, the conditions assumed on the vector field and
the standard theory of FDEs ensure that the solutions of the variational equation
also define the continuous skew-product semiflow
\[
 \wit\Pi_L\colon\R^+\!\times \mK\times C\to\mK\times C\,,\quad
 (t,\w,x,v)\mapsto(\Pi(t,\w,x),w(t,\w,x,v))\,,
\]
where $w(t,\w,x,v)$ represents the same function as above.
We show that, given $(\w,x)\in\mK$ and $T\geq r$, the map
$C\to\WW$, $v\mapsto w(T,\w,x,v)$
is continuous, and that it is also compact if $T\geq 2r$. This property is the main
tool in the proof of a result which will be fundamental in the paper: if we follow
the classical way to define the upper Lyapunov exponent of $\mK$ with respect to
the pseudo-continuous flow $\Pi_L$, it agrees with the (classical) one with respect to
$\wit{\Pi}_L$.
\par
In Section \ref{sec4} we strength slightly the Lipschitz conditions assumed on $\tau$,
and consider a set $\mK$ as described above, with the additional property that
it projects over the whole base.
Let $\lambda_\mK$ be its upper Lyapunov exponent. We prove that the
condition $\lambda_{\mK}<0$ is equivalent
to the exponential stability of $\mK$ for the usual Lipschitz norm, and also to the
exponential stability of $\mK$ expressed in terms of the supremum norm.
This extends to our nonautonomous setting results previously proved by Hartung in
\cite{hart2} in the case of periodic SDDEs.
\par
Section~\ref{sec5} considers again the initial conditions assumed on
$\tau$, and contains the adequate version of the characterization of the exponential stability.
The results are very similar to that of Section~\ref{sec4}: the only difference
relies in the expression of the exponential stability in terms of the norm in $C$.
In this less restrictive setting, we go further in the analysis.
We prove that, if base $\W$ is minimal and $\lambda_{\mK}<0$,
then $\mK$ is an $m$-cover of the base flow
$(\W,\sigma,\R)$ admitting a flow extension. We also establish several
properties on its domain of attraction. The paper is completed with the
following nice extension: if $\mP$ is a positively $\Pi$-invariant compact set
such that $\lambda_{\mM}<0$ for every minimal set $\mM\in\mP$,
then $\mP$ only contains a finite number of minimal sets; and, in addition,
the subsets of $\mP$ determined by its intersection with the domains of attraction
of its minimal subsets agree with the connected components of $\mP$.
A conclusion of all the preceding results closes the paper:
the existence of a uniformly exponentially stable solution
of a single uniformly almost periodic SDDE ensures the existence of
exponentially stable almost periodic solutions.
\par
We close this introduction by pointing out that the conclusions of this paper
provide the tools to develop appropriate versions for the context of nonautonomous
SDDEs of some applied models described by
Arino {\em et al.} \cite{arsa}, Smith \cite{smit},
Wu \cite{wu}, Hartung {\em et al.}~\cite{hkvv},
Novo {\em et al.} \cite{noos4}, Insperger and St\'{e}p\'{a}n \cite{inst},
and some of the references therein. In particular, the results of this paper
are the key point in the extension to the case of nonautonomous SDDEs
of the results about exponential stability for biological
neural networks of~\cite{noos4}, which will be developed elsewhere.
\section{Some preliminaries}\label{sec2}
In this section we introduce the basic notions of topological dynamics which will be
used throughout the paper. They can be found in Sacker and
Sell~\cite{SackerSell,SackerSellDich}, Chow and Leiva \cite{chle1,chle2},
Shen and Yi \cite{shyi}, and references therein.
\par
Let $\W$ be a complete metric space. A ({\em real, continuous}) {\em flow\/} $(\W, \sigma, \R)$
is defined by a continuous map $\sigma\colon\R\times\W\to\W$, $(t,\w)\mapsto\sigma(t,\w)$
satisfying
\begin{itemize}
  \item[\hypertarget{f1}(f1)] $\sigma_0=\text{Id}$,
  \item[\hypertarget{f2}(f2)] $\sigma_{t+s}=\sigma_{t}\circ\sigma_{s}$ for all $t,s\in\R$,
\end{itemize}
where $\sigma_{t}(\w)=\sigma(t,\w)$ for all $t\in\R$ and $\w\in\W$. The set
$\{\sigma_{t}(\w)\,|\,t\in\R\}$ is called the {\em orbit\/} of the point $\w\in\W$.
A subset $\mM\subseteq\W$ is $\sigma$-{\em invariant\/}
(or just {\em invariant\/}) if $\sigma_{t}(\mM)\subseteq\mM$ for every $t\in\R$
(which clearly ensures that $\sigma_t(\mM)=\mM$ for every $t\in\R$).
A subset $\mM\subseteq\W$ is called {\em minimal\/} if it is compact, $\sigma$-invariant,
and its only nonempty compact $\sigma$-invariant subset is itself. Zorn's lemma ensures
that every compact and $\sigma$-invariant set contains a minimal subset. Note that
a compact $\sigma$-invariant subset is minimal if and only if each one of its orbits is dense.
We say that the continuous flow $(\W, \sigma, \R)$ is {\em recurrent\/} or {\em minimal\/}
if $\W$ itself is minimal. The flow is {\em local\/}
if the map $\sigma$ is defined, continuous, and satisfies \hyperlink{f1}{(f1)} and
\hyperlink{f2}{(f2)} (this last one whenever it makes sense) on an open subset
$\mO\subseteq\R\times\W$ containing $\{0\}\times\W$.
And, in the case of a compact base $\W$,
the flow $(\W, \sigma, \R)$ is {\em almost periodic\/} if for every
$\ep>0$ there exists $\delta=\delta(\ep)>0$ such that, if $\w_1,\w_2\in\W$ satisfy
$d_\W(\w_1,\w_2)<\delta$ (where $d_\W$ is the distance on $\W$), then
$d_\W(\sigma_{t}(\w_1), \sigma_{t}(\w_2))<\ep$ for all $t\in\R$.
\par
As usual, we represent $\R^\pm=\{t\in\R\,|\,\pm t\geq 0\}$.
If $\sigma\colon\R^+\!\times\W\to\W$, $(t,\w)\mapsto\sigma(t,\w)$ is a continuous map
which satisfies the properties \hyperlink{f1}{(f1)} and  \hyperlink{f2}{(f2)}
described above for all $t,s\in\R^+$, then
$(\W,\sigma,\R^+)$ is a ({\em real, continuous}) {\em semiflow\/}.
The set $\{\sigma_{t}(\w)\,|\,t\ge 0\}$ is the ({\em positive}) {\em semiorbit\/} of the point
$\w\in\W$. If this semiorbit is relatively compact,
the {\em omega-limit set $\mO(\w)$ of the point $\w\in\W$} (or {\em of its semiorbit})
is the set of limits of sequences of the form
$(\sigma_{t_m}(\w))$ with $(t_m)\uparrow\infty$.
A subset $\mM\subseteq\W$ is {\em positively $\sigma$-invariant\/}
(or just {\em $\sigma$-invariant\/}, or {\em invariant}) if $\sigma_{t}(\mM)\subseteq \mM$
for all $t\geq0$. This is the case of all the omega-limit sets.
A positively $\sigma$-invariant compact set $\mM$ is {\em minimal\/} if it
does not contain properly any positively $\sigma$-invariant compact set. If $\W$ is minimal,
we say that the semiflow is {\em minimal}. The semiflow is {\em local\/} if the map $\sigma$
is defined, continuous, and satisfies \hyperlink{f1}{(f1)} and  \hyperlink{f2}{(f2)}
on an open subset $\mO\subseteq\R^+\!\times\W$
containing $\{0\}\times\W$. In this case, the definitions of positively invariant set and
minimal set are the same as above. In particular, they are composed of globally defined
positive semiorbits, so that the restriction of the semiflow to one of these sets is global.
Note that, in the local case, we need to be sure that a semiorbit is (at least)
globally defined in order to talk about its omega-limit set.
\par
A continuous semiflow $(\W,\sigma,\R^+)$ {\em admits a continuous
flow extension\/} if there exists a
continuous flow $(\W,\overline{\sigma},\R)$ such that $\overline{\sigma}(t,\w)=\sigma(t,\w)$
for all $t\in\R^+$ and $\w\in\W$. Let $\mM$ be a positively $\sigma$-invariant compact set.
A point $\w\in \mM$ {\em admits a backward extension in $\mM$\/} if there exists a
continuous map $\theta_{\w}\colon\R^{-}\!\to \mM$ such that $\theta_{\w}(0)=\w$ and
$\sigma(t, \theta_{\w}(s))=\theta_{\w}(t+s)$ whenever $s\le-t\le0$. We will
use the words \lq \lq admits {\em at least} a backward extension in $\mM$\/" to
emphasize the fact that the extension may be non unique.
The set $\mM$ {\em admits a continuous
flow extension\/} if the semiflow restricted to it admits one.
It is known that, if the semiorbit of a point $\w\in\W$
is relatively compact, then any element of the omega-limit set $\mO(\w)$
admits at least a backward extension in $\mO(\w)$
(see Proposition II.2.1 of~\cite{shyi}); and that, in the case that $\W$ is locally
compact, the existence of a continuous flow
extension for $\mM$ is equivalent to the existence and uniqueness of a backward extension
for each of its points (see in Theorem II.2.3 of~\cite{shyi}).
\par
A (local or global, continuous) semiflow is of
{\em skew-product type\/} when it is defined on a vector
bundle and has a triangular structure. More precisely, let $(\W,\sigma,\R^+)$ be a global
semiflow on a {\em compact\/} metric space $\W$, and let $X$ be a Banach space.
We will represent $\wt=\sigma_t(\w)=\sigma(t,\w)$. A local semiflow ($\W\times X, \Pi, \R^+$)
is a {\em skew-product semiflow with base\/} $(\W,\sigma,\R)$ {\em and fiber $X$}
if it takes the form
\begin{equation}\label{2.1skew}
\Pi\colon\mU\subseteq\R^+\!\times\W\times X\to\W\times X\,,
\quad (t,\w, x)\mapsto(\wt,u(t,\w, x))\,.
\end{equation}
Property \hyperlink{f2}{(f2)}
means that the map $u$ satisfies the {\em cocycle property\/}
$u(t+s,\w,x)=u(t,\ws, u(s,\w,x))$ whenever
the right-hand function is defined. It is frequently assumed that the base semiflow is in fact
a flow. We will add explicitly this hypothesis when we use it.
\par
Now we state some definitions about stability. All of them refer to properties of the
skew-product semiflow $\Pi$ defined by \eqref{2.1skew}. The norm on $X$ and the corresponding
distance are represented by $\n\pu _X$ and $d_X$. A compact set $\mK\subset\W\times X$
{\em projects over the whole base} if for any $\w\in\W$ there exists $x\in X$ such that
$(\w,x)\in\mK$. This is the type of sets on which the concept of stability make sense.
Note that this is always the case if $\mK$ is positively $\Pi$-invariant and $\W$
is minimal.
\begin{defi}\label{2.defi1}
A positively $\Pi$-invariant compact set $\mK\subset\W\times X$ projecting over the whole
base is {\em uniformly stable} if for
any $\ep>0$ there exists $\delta(\ep)>0$, such that, if the points $(\w,\bar x)\in \mK$ and
$(\w,x)\in\W\times X$ satisfy $d_X(x, \bar x)<\delta(\ep)$, then $u(t,\w,x)$ is defined
for $t\in[0,\infty)$ and $d_X(u(t,\w, x), u(t,\w,\bar x))\le\ep$ for all $t\geq0$.
The restricted semiflow $(\mK, \Pi, \R^+)$ is said to be {\em uniformly stable}.
\end{defi}
\begin{defi}\label{2.defi2}
A positively $\Pi$-invariant compact set $\mK\subset\W\times X$ projecting over the whole
base is {\em uniformly asymptotically stable}
if it is uniformly stable and, in addition, there exists $\delta>0$ such that, if
the points $(\w,\bar x)\in \mK$ and  $(\w,x)\in\W\times X$ satisfy $d_X(x, \bar x)<\delta$,
then $u(t,\w,x)$ is defined for $t\in[0,\infty)$  and
$\lim_{t\to\infty}d_X(u(t,\w, x),u(t,\w,\bar x))=0$ uniformly in $(\w,\bar x)\in \mK$.
The restricted semiflow $(\mK, \Pi, \R^+)$ is said to be {\em uniformly asymptotically stable}.
\end{defi}
\begin{defi}\label{2.defi3}
A positively $\Pi$-invariant compact set $\mK\subset\W\times X$ projecting over the whole
base is {\em exponentially stable} if there exist $\delta_0>0$, $C>0$ and $\alpha>0$,
such that, if the points $(\w,\bar x)\in \mK$ and
$(\w,x)\in\W\times X$ satisfy $d_X(x, \bar x)<\delta_0$,
then $u(t,\w,x)$ is defined for $t\in[0,\infty)$ and
$d_X(u(t,\w, x), u(t,\w,\bar x))\le C\,e^{-\alpha·t}\,d_X(x, \bar x)$ for all $t\geq0$.
The restricted semiflow $(\mK, \Pi, \R^+)$ is said to be {\em exponentially stable}.
\end{defi}
The next definitions and properties
refer to the special case of a {\em linear\/} skew-product semiflow.
A global continuous skew-product semiflow $\Pi$ is {\em linear\/} if it takes the form
\begin{equation}\label{2.defpilin}
 \Pi\colon\R^+\!\times\W\times X\to\W\times X\,,\quad
 (t,\w, x)\mapsto(\wt,\phi(t,\w)\,x)\,,
\end{equation}
where $\phi(t,\w)$ is a bounded linear operator on $X$; in other words, if
$u(t,\w,x)$ is linear in $x$ for each $(t,\w)\in\R^+\!\times\W$.
In what follows, we assume that the base $(\W,\sigma,\R)$ is a flow (not just a semiflow)
on a compact metric space. This hypothesis will be weakened later:
see Remark \ref{2.notalifting}.
\begin{defi}\label{2.defLyap}
The {\em upper Lyapunov exponent $\lambda^+_{s}(\w)$ of $\w\in\W$ for
the semiflow $(\W\times X,\Pi,\R^+)$} given by \eqref{2.defpilin} is
\[
 \lambda^+_{s}(\w):=\sup_{x\in X,\,x\ne 0}\lambda^+_{s}(\w,x)\,,
\]
where\vspace{-.1cm}
\[
\lambda^+_{s}(\w,x):=\limsup_{t\to\infty}\frac{1}{t}\:\ln\n{\phi(t,\w)\,x}_X\,;
\]
and the {\em upper Lyapunov exponent of the set $\W$ for the
semiflow $(\W\times X,\Pi,\R^+)$} is
\[
\lambda_{\W}:=\sup_{\w\in\W}\lambda^+_{s}(\w)\,.
\]
\end{defi}
Proposition 2.1 of \cite{chle2} proves that $\lambda_\W<\infty$, and
Theorem 4.2 of \cite{chle1} shows that
\begin{equation}\label{2.lyap}
 \lambda^+_{s}(\w)=\limsup_{t\to\infty}\frac{1}{t}\:\ln\n{\phi(t,\w)}_{\Lin(X,X)}\,.
\end{equation}
In addition,
Proposition II.4.1 and Corollary II.4.2 of~\cite{shyi} show that
\begin{equation}\label{2.lyap2}
 \forall\;\mu>\lambda_\W \;\;\exists\;
 k_\mu\ge 1\;\text{ such that }\;\n{\phi(t,\w)}_{X}\le k_\mu\,e^{\mu\,t}
 \;\;\forall\;\w\in\W\,.
\end{equation}
\begin{nota}\label{2.notalifting}
We will very often work with a linear skew-product semiflow
\[
 \Pi\colon\R^+\!\times\W\times X\to\W\times X\,,\quad
 (t,\w, x)\mapsto(\wt,\phi(t,\w)\,x)\,,
\]
for which the base $(\W,\sigma,\R^+)$ is a global semiflow on a compact
metric space, with the fundamental
property that each one of its elements admits at least
a backward extension in $\W$. (Recall that this is the situation at least in the case that
$\W$ is minimal, which we do not assume in what follows.)
Our next purpose is to show that the previous
definitions of Lyapunov exponents and the properties that we will require
make sense also in this setting, in which the existence of a flow extension
on $\W$ is not required.
Part of the argument is taken from Section II.2.2 of~\cite{shyi} and from
Theorem 10 of Chapter 4 of \cite{SackerSell}. Let us define
\[
 \W^*=\{\xi\in C(\R,\W)\,|\;\sigma(t,\xi(s))=\xi(t+s)\text{ for $t\ge 0$ and $s\in\R$}\}\,;
\]
that is, the elements of $\W^*$ are the global orbits provided by all the backward extensions
of all the elements of $\W$.
Then $\W^*$ is a compact subset of $C(\R,\W)$ for the compact-open
topology of $C(\R,\W)$, which agrees with the topology given by the distance
\[
 d_{C(\R,\W)}(\xi_1,\xi_2):=\Sum_{m=1}^\infty \,
 \frac{1}{2^m}\;\max_{s\in[-m,m]} d_\W (\xi_1(s),\xi_2(s))\,;
\]
that is, $\W^*$ is a compact metric space.
Note that we have assumed that for every $\w\in\W$ there exists
at least a point $\xi\in\W^*$ with $\xi(0)=\w$. As said before, it is proved in
Theorem II.2.3 of \cite{shyi}
that this correspondence is one-to-one if and only if the semiflow $(\W,\sigma,\R^+)$
admits a continuous flow extension.
In this more general setting, it is also possible to define a continuous flow on the
set $\W^*$,
called {\em the lifting flow\/}, which, roughly speaking, projects onto $\W$.
It is given by
$\sigma^*\colon\R\times\W^*\to\W^*,\;(t,\xi)\mapsto \xi\pu t$, with
$(\xi\pu t)(s)=\xi(t+s)$. Hence, whenever
$\w=\xi(0)$ we have, for $t\ge 0$,
\[
 \wt=\sigma(t,\w)=\sigma(t,\xi(0))=\xi(t)=\sigma^*(t,\xi)(0)=(\xi\pu t)(0)\,.
\]
Now we can define
\[
 \Pi^*\colon\R^+\!\times\W^*\times X\to\W^*\times X\,,\quad
 (t, \xi, x)\mapsto(\xi\pu t, \phi^*(t,\xi)\,x)=(\xi\pu t, \phi(t,\xi(0))\,x)\,,
\]
which is a continuous linear skew-product semiflow with base flow $(\W^*,\sigma^*)$, and
define the corresponding upper Lyapunov exponent $(\lambda^*)^+_s(\xi)$ for $\xi\in\W^*$ and
$\lambda^*_{\W^*}:=\sup_{\xi\in\W^*}(\lambda^+)^*_{s}(\xi)$.
It is clear that $(\lambda^*)^+_s(\xi)$ only depends on $\xi(0)$, which belongs to $\W$.
In other words, we can define $\lambda_s^+(\w)$ and
$\lambda_\W$ directly from $\Pi$, as in Definition~\ref{2.defLyap},
and then we have $(\lambda^*)^+_s(\xi)=\lambda_s^+(\w)$ for $\w=\xi(0)$, and
$\lambda_\W=\lambda^*_{\W^*}$.
And it is clear that \eqref{2.lyap} and \eqref{2.lyap2} are still valid.
\par
Note finally that, if $\W$ is minimal, then $\W^*$ is also minimal. In order to prove this
assertion,
we must take $\xi_0,\,\xi$ in $\W^*$, and find a sequence
$(t_m)$ in $\R$ such that $\xi_0\pu t_m$ converges to $\xi$ uniformly on $[-k,k]$ 
for all $k>0$.
Let us take $\w=\xi(-k)$ and $\w_0=\xi_0(-r)$, use the minimality of $\W$ to take
a sequence $(t_m)$ in $\R^+$ with $\w=\lim_{m\to\infty}\w_0\pu t_m$, and deduce from the uniform continuity of $\sigma$ on $\W\times[0,2k]$ that
$\wt=\lim_{m\to\infty}\w_0\pu(t_m+t)$ uniformly on $t\in[0,2r]$; that is,
$\xi(t-k)=\lim_{m\to\infty}(\xi_0\pu t_m)(t-k)$ uniformly on $t\in[0,2k]$,
which is the sought-for property.
\end{nota}
We complete this section by fixing some notation which will be used throughout the paper.
Given two Banach spaces $(X,\n{\pu }_X)$ and $(Y,\n{\pu }_Y)$,
$\Lin(X,Y)$ represents the set of bounded linear maps $\phi\colon X\to Y$ equipped
with the operator norm $\n{\phi}_{\Lin(X,Y)}=\sup_{\n{x}_X=1}\n{\phi(x)}_Y$.
The maximum delay of the equations that we will consider is represented by $r>0$.
The set $C$ represents the Banach space of continuous functions
$C([-r, 0], \R^n)$ equipped with the norm $\n{\psi}_{C}:=\sup_{s\in[-r, 0]}|\psi(s)|$,
where $|\cdot|$ represents the Euclidean norm in $\R^n$. The subset $C^1\subset C$ is
given by the functions which have continuous derivative on $[-r,0]$ (one-sided
derivatives at the end points of the interval). The set $L^\infty$ is
the space of Lebesgue-measurable functions $\psi\colon[-r, 0]\to\R^n$
which are {\em essentially bounded\/}, which means that there exists $k\ge 0$ such
that the set $\{x\in[-r, 0]\:|\;|\psi(x)|>k\}$ has zero measure. The norm on $L^\infty$,
which is defined as the inferior of the set of real numbers $k\ge 0$ with the previous
property, is denoted by $\n\pu _{L^\infty}$. The set $\WW$ is the Banach space of
Lipschitz-continuous functions $\psi\colon[-r, 0]\to\R^n$ equipped with the
Lipschitz norm $\n{\psi}_{\WW}:=\max\{\n{\psi}_C, \n{\dot{\psi}}_{L^\infty}\}$.
Note that Arzel\'{a}--Ascoli theorem ensures that any bounded set of $\WW$ is
relatively compact in $C$.
Finally, given a continuous function $x\colon[-r, \gamma]\to\R^{n}$ for $\gamma>0$ and a time
$t\in[0,\gamma]$, we denote by $x_{t}\in C$ the function defined by
$x_{t}(s)=x(t+s)$ for $s\in[-r, 0]$.
\section{FDEs with state-dependent delay} \label{sec3}
Let $(\W, \sigma, \R)$ be a continuous flow on a compact metric space.
As in the previous section, we write $\wt=\sigma(t,\w)$ for $t\in\R$ and $\w\in\W$.
Given $F\colon\W\times\R^{n}\times\R^{n}\to\R^{n}$ and $\tau\colon\W\times C\to[0, r]$,
we consider the family of nonautonomous SDDEs
\begin{equation}\label{3.eq}
\dot{y}(t)=F(\wt,y(t), y(t-\tau(\wt,y_{t})))\,,\qquad t\geq 0\,,
\end{equation}
for $\w\in\W$.
All or part of the following conditions will be assumed on $F$ and $\tau$:
\begin{itemize}
  \item[H1] \hypertarget{3.H1}
  $F\colon\W\times\R^{n}\times\R^{n}\to\R^{n}$ is continuous,
  and its partial derivatives w.r.t.~its second and third arguments
  exist and are continuous on $\W\times\R^{n}\times\R^{n}$.
  In particular,
  the functions $D_iF\colon \W\times\R^{n}\times\R^{n}\to\Lin(\R^n,\R^n)$
  exist and are continuous for $i=2,3$.
  \item[H2] 
   \begin{itemize}
   \item[(1)] \hypertarget{3.H21}
   $\tau\colon\W\times C\to[0,r]$ is continuous and differentiable
   w.r.t.~its second argument, with $D_2\tau\colon\W\times C\to\Lin(C, \R)$ continuous.
   \item[(2)] \hypertarget{3.H22} $D_2\tau$ is locally Lipschitz-continuous
   in the following sense:
   for every compact subset $\mK\subset \W\times C$ there exists a constant
   $L_2=L_2(\mK)>0$ such that
       \[
       \qquad\n{D_2\tau(\w,x_1)-D_2\tau(\w,x_2)}_{\Lin(C, \R)}
       \le L_2\n{x_1-x_2}_{C}
       \]
       for all $(\w,x_1)$ and $(\w,x_2)$ in $\mK$.
   \end{itemize}
\end{itemize}
\begin{nota}
Note that \hyperlink{3.H21}{H2(1)} ensures the next property:
\begin{itemize}
\item[H2]
\begin{itemize}
   \item[(3)] \hypertarget{3.H23}
   $\tau$ is locally Lipschitz-continuous in this sense: for every
   compact subset $\mK\subset \W\times C$ there exists a constant $L_1=L_1(\mK)>0$ such
   that
       \[
       |\tau(\w,x_1)-\tau(\w,x_2)|\le L_1\n{x_1-x_2}_{C}
       \]
       for all $(\w,x_1)$ and $(\w,x_2)$ in $\mK$.
\end{itemize}
\end{itemize}
In order to prove this assertion, we take a compact subset
$\mK\subset\W\times C$ and note that the set
$\bar\mK=\{(\w,s\,x_1+(1-s)\,x_2)\,|\;(\w,x_1),\,(\w,x_2)\in\mK \text{ and }s\in[0,1]\}$
is also compact in $\W\times C$.
We define $L_1=L_1(\mK):=\sup\{\n{D_2\tau(\w,\bar x)}_{\Lin(C,\R)}\,|
\;(\w,\bar x)\in\bar \mK\}$. Then,
\[
 |\tau(\w,x_1)-\tau(\w,x_2)|\le \left|\int_0^1 D_2\tau(\w,s\,x_1+(1-s)\,x_2)
 (x_1-x_2)\,ds\right|
 \le L_1\n{x_1-x_2}_C
\]
whenever $\w\in\W$ and $x_1$, $x_2\in \mK$, as asserted.
\end{nota}
Let us now summarize the most basic properties of the solutions of the equation~\eqref{3.eq}
ensured by hypotheses \hyperlink{3.H1}{H1} and \hyperlink{3.H21}{H2(1)}.
In the statement of the next theorem a fundamental role is played by
the set of pairs \lq\lq(equation, initial datum)"
which satisfy the compatibility condition given by the vector field; namely
\begin{equation}\label{3.defC0}
 \mC_0=\{(\w,x)\in\W\times C^1\,|\;\dot{x}(0^+)=F(\w,x(0), x(-\tau(\w,x)))\}\,.
\end{equation}
The next result, strongly based on previous properties proved in \cite{hart1}, is proved in
Theorem 3.3 and Corollary 3.4 of \cite{mano1}.
\begin{teor}\label{3.flujocontinuo}
Suppose that conditions {\rm \hyperlink{3.H1}{H1}} and {\rm \hyperlink{3.H2}{H2(1)}} hold.
Then,
\begin{itemize}
\item[(i)] for $\w\in\W$ and $x\in\WW$, there exists a unique
 maximal solution $y(t,\w,x)$ of the equation \eqref{3.eq} corresponding to $\w$
 satisfying $y(s,\w,x)=x(s)$ for $s\in[-r,0]$, which is defined for
 $t\in[-r,\beta_{\w,x})$ with $0<\beta_{\w,x}\le\infty$. In particular,
 $y(t,\w,x)$ is continuous on $[-r,\beta_{\w,x})$ and satisfies~\eqref{3.eq} on
 $(0,\beta_{\w,x})$, and there exists the
 lateral derivative $\dot y(0^+\!,\w,x)=F(\w,x(0),x(-\tau(\w,x))$.
\end{itemize}
Let us define $u(t,\w,x)(s):=y(t+s,\w,x)$
for $(\w,x)\in\W\times\WW$, $t\in[0,\beta_{\w,x})$, and $s\in[-r,0]$. Then,
\begin{itemize}
\item[(ii)] $u(t,\w,x)\in\WW$ for all $t\in[0,\beta_{\w,x})$.
\item[(iii)] If $\sup_{t\in[0,\beta_{\w,x})}
 \n{u(t,\w,x)}_C<\infty$ then $\beta_{\w,x}=\infty$ and, in addition,
 the set $\{(\wt,u(t,\w,x))\,|\;t\in[r,\infty)\}$ is relatively compact
 in $\W\times\WW$.
\end{itemize}
Let us further define $\mC_0\subset\W\times\WW$ by \eqref{3.defC0} and
\begin{align}
 &\mU:=\{(t,\w,x)\,|\;(\w,x)\in\W\times\WW,\;t\in[0,\beta_{\w,x})\}
 \subset\R^+\!\times\W\times\WW,\nonumber\\
 &\Pi\colon\mU\to\W\times W^{1,\infty}\,,\quad (t,\w,x)\mapsto(\wt,u(t,\w,x))\,,
 \label{3.defPi}\\
 &\wit\mU:=\{(t,\w,x)\in\mU\,|\;t\ge r\}\subset\R^+\!\times\W\times\WW,\nonumber\\
 &\mU^{\,0}:=\{(t,\w,x)\,|\;(\w,x)\in\mC_0,\;t\in[0,\beta_{\w,x})\}\subset
 \R^+\!\times\W\times\WW,\nonumber
\end{align}
and provide $\mU$, $\wit\mU$, $\mC_0$ and $\mU^{\,0}$
with the respective subspace topologies. Then,
\begin{itemize}
\item[(iv)] the set $\mU$ is open in $\R^+\!\times\W\times\WW$ and
 $\Pi$ satisfies conditions {\rm \hyperlink{(1)}{(f1)}} and {\rm \hyperlink{(2)}{(f2)}}
 of Section {\rm \ref{sec2}} (wherever it makes sense, and with $\W$
 replaced by $\W\times\WW$).
\item[(v)] The map $\mU\to\W\times C\,,\;
 (t,\w,x)\mapsto(\wt,u(t,\w,x))$ is continuous.
\item[(vi)] The map
 $\wit\mU\to\W\times\WW,\;(t,\w,x)\mapsto(\wt,u(t,\w,x))$ is continuous.
\item[(vii)] Let us fix $\wit t\ge 0$ with $\mU_{\,\wit t}:=\{(\w,x)\,|\;(\wit t,\w,x)\in\mU\}$
 nonempty. Then the map
 $\mU_{\,\wit t}\to\W\times\WW,\;(\w,x)\mapsto
 (\w\pu\wit t,u(\wit t,\w,x))$ is continuous.
\item[(viii)] The map
 $\mU^{\,0}\to\mC_0\subset\W\times\WW,\;(t,\w,x)\mapsto (\wt,u(t,\w,x))$ is continuous.
\item[(ix)] Let $\mK\subset\W\times\WW$
 be a positively $\Pi$-invariant compact set. Then the restriction
 of $\Pi$ to $\mK$ defines a global continuous semiflow on $\mK$.
\end{itemize}
\end{teor}
Note that point (i) states that
\begin{equation}\label{3.dery}
 \dot y(t,\w,x)=F(\wt,y(t,\w,x),y(t-\tau(\wt,u(t,\w,x))))
 \quad\text{for all $t\in[0,\beta_{\w,x})$}\,,
\end{equation}
where the derivative at $t=0$ must be understood as the right-hand derivative.
\begin{nota}\label{3.notamismas}
As anticipated in the Introduction, we will say that $\Pi$ is a {\em pseudo-continuous\/}
semiflow. The definitions of semiorbit,
positively $\Pi$-invariant set and of minimal set are the same.
Note that the positively
$\Pi$-invariance of a set $\mM$ ensures that $\R^+\!\times\mM\subseteq\mU$.
If $\mK$ is a positively $\Pi$-invariant compact set $\mK$, then also the definition
of existence of backward extension of its element $(\w,x)$ in $\mK$ is the same.
In addition, if a point
$(\w,x)$ has bounded $\Pi$-semiorbit (which ensures that $\beta_{\w,x}=\infty$ and
that $\{(\wt,u(t,\w,x))\,|\;t\in[r,\infty)\}\subset\W\times\WW$ is relatively compact),
we can define its omega-limit set as in Section \ref{sec2}:
Theorem~\ref{3.teorexisteM}(ii) will show that this causes no confusion.
Finally, also Definitions \ref{2.defi1}, \ref{2.defi2} and \ref{2.defi3} can
be directly adapted to $\Pi$.
\end{nota}
In most of this section, we will be working with a subset $\mK$ of
$\W\times\WW$ satisfying the following conditions
(see Section~\ref{sec2} and Remark~\ref{3.notamismas}):
\begin{hipos}\label{3.hipos}
Conditions \hyperlink{3.H1}{H1} and \hyperlink{3.H21}{H2(1)} hold,
and $\mK\subset\W\times\WW$ is
a positively $\Pi$-invariant compact set such that each one of its elements admits
a backward extension in $\mK$.
\end{hipos}
\begin{nota}\label{3.notapropK}
If Hypotheses \ref{3.hipos} hold, then
the semiflow $(\mK,\Pi,\R^+)$ is globally defined and continuous:
see Theorem \ref{3.flujocontinuo}(ix), and note that
we denote with the same symbol $\Pi$ the restriction $\Pi|_\mK$.
In addition, the existence of backward extension in $\mK$
of its elements ensures that $\mK\subset\mC_0$, where $\mC_0$ is defined by \eqref{3.defC0}.
\end{nota}
Such a set $\mK$ will be fixed once we have proved the next theorem.
It shows that any positively $\Pi$-invariant bounded set determines a
positively $\Pi$-invariant compact set; and it explains that each positively
$\Pi$-invariant compact set contains a maximal subset $\mK$ satisfying the conditions of
Hypotheses \ref{3.hipos}.
\begin{teor}\label{3.teorexisteM}
Suppose that conditions {\rm \hyperlink{3.H1}{H1}} and {\rm \hyperlink{3.H21}{H2(1)}} hold,
and let $\Pi$ be defined by~\eqref{3.defPi}.
\begin{itemize}
\item[\rm(i)] If $\mK_0\subset\W\times\WW$ is
a positively $\Pi$-invariant bounded set, then the set
\[
 \mK_1=\text{\rm closure\,}_{\W\times\WW}\{\Pi(2r,\w,x)\,|\;(\w,x)\in\mK_0\}
\]
is a positively $\Pi$-invariant compact set.
\item[\rm(ii)] Let $(\wit\w,\wit x)\in\W\times\WW$ have bounded semiorbit. Then its omega-limit
set $\mO(\wit\w,\wit x)$ is well-defined, positively $\Pi$-invariant, and compact. In addition,
any point $(\w,x)\in\mO(\wit\w,\wit x)$ admits at least a backward extension in
$\mO(\wit\w,\wit x)$.
\item[\rm(iii)] If $\mK_2$ is a positively $\Pi$-invariant compact set, then the set
\[
 \mK_3=\{(\w,x)\in\mK_2\,|\;(\w,x)\text{ admits a backward extension in }\mK_2\}
\]
is a nonempty positively $\Pi$-invariant compact set, and is the maximal
subset of $\mK_2$ with these properties.
\end{itemize}
\end{teor}
\begin{proof}
(i) The positively $\Pi$-invariance of $\mK_1$ follows easily from
Theorem~\ref{3.flujocontinuo}(vii). Therefore,
it suffices to show that given any sequence $((\w_m,x_m))$ in $\mK_0$,
the sequence $(\Pi(2r,\w_m,x_m))=((\w_m\pu(2r),u(2r,\w_m,x_m)))$
admits a subsequence which converges to a point
$(\w^*,x^*)\in\W\times\WW$. We can assume without restriction that
there exists $\w=\lim_{n\to\infty}\w_m$, so that
$\w^*=\w\pu(2r)$.
We represent $y_m\colon[-r,2r]\to\R^n,\;t\mapsto y(t,\w_m,x_m)$ and
note that $(\w_m\pu t,(y_m)_t)\in\mK_0$ for all $t\in[0,2r]$.
Since $\mK_0$ is bounded in $\W\times\WW$,
the sequence $(y_m)$ is uniformly bounded in $C([-r,2r],\R^n)$.
In addition,
\begin{equation}\label{3.deriv1}
 \dot{y}_m(t)=F(\w_m\pu t,y_m(t),y_m(t-\tau(\w_m\pu t,(y_m)_{t})))
\end{equation}
for $t\in[0,2r]$. The bound of $\mK_0$ together with
\hyperlink{3.H1}{H1}, Theorem~\ref{3.flujocontinuo}(v) and \hyperlink{3.H21}{H2(1)},
ensures that the sequence $(\dot{y}_m)$ is also contained in
$C([0,2r],\R^n)$ and is uniformly bounded on $[0,2r]$. Therefore,
Arzel\'{a}--Ascoli theorem
provides a subsequence $(y_k)$ which converges uniformly on $[0,2r]$
to a function $y^*\in C([0,2r],\R^n)$. In addition, hypotheses \hyperlink{3.H21}{H2(1)}
ensures that the sequence $(t\mapsto \tau(\w_k\pu t,(y_k)_t))$ converges
to the function $t\to\tau(\w\pu t,(y^*)_t)$ uniformly on $[r,2r]$.
Therefore,
\begin{equation}\label{3.deriv2}
 \lim_{k\to\infty}F(\w_k\pu t, y_k(t),y_k(t-\tau(\w_k\pu t,(y_k)_{t})))
 =F(\w\pu t, y^*(t), y^*(t-\tau(\w\pu t,(y^*)_t)))
\end{equation}
uniformly on $[r,2r]$. In turn, this property ensures that
$(y_k(r)+\int_r^t F(\w_k\pu s, y_k(s),$ $
y(s-\tau(\w_k\pu s,(y_k)_s),\w_k,x_k))\,ds)$ (which, by \eqref{3.deriv1},
agrees with the sequence $(y_k(t))$)
converges to the point
$y^*(r)+\int_r^t F(\w\pu s, y^*(s), y^*(s-\tau(\w\pu s,(y^*)_s)))\,ds$
for $t\in[r,2r]$.
Since $\lim_{k\to\infty} y_k(t)=y^*(t)$
for $t\in[0,2r]$, we see that there exists
$\dot y^*(t)$ for $t\in[r,2r]$ and that it agrees with
$F(\wt,y^*(t),y^*(t-\tau(\wt,(y^*)_t))$. This fact together with
\eqref{3.deriv1} and \eqref{3.deriv2} shows that
$(\dot{y}_k)$ converges to $(\dot{y}^*)$ uniformly on $[r,2r]$.
Altogether, we see that the sequence $(u(2r,\w_k,x_k))$ converges to
$x^*:=y^*_{2r}$ in $W^{1,\infty}$, which proves (i).
\smallskip\par
(ii) Theorem~\ref{3.flujocontinuo}(iii) shows that the
classical definition of omega-limit set $\mO(\w,x)$ of a point $(\w,x)$  with bounded
$\Pi$-semiorbit makes sense. Since it agrees with the omega-limit
set of the point $\Pi(2r,\w,x)$, we can adapt the proof of (i) to show that
$\mO(\w,x)$ is compact. Its positively $\Pi$-invariance follows from
Theorem~\ref{3.flujocontinuo}(vii). Theorem~\ref{3.flujocontinuo}(ix)
ensures that the restricted semiflow $(\mO,\Pi,\R^+)$ is continuous, and hence
Proposition II.2.1 of~\cite{shyi} proves the last assertion in (ii).
\smallskip\par
(iii) It is clear that the set $\mK_3$ is a positively $\Pi$-invariant subset of $\mK_2$.
Since $\mK_2$ contains at least a minimal subset, point (ii) ensures that
the set $\mK_3$ is nonempty.
Therefore, since $\mK_2$ is compact, the goal is to check that $\mK_3$ is closed. Let us fix a
point $(\w,x)\in\text{closure\,}_{\W\times\WW}\mK_3$. We will follow an
iterative procedure. The first step is to find a point
$(\w\pu(-1),x_{-1})\in\text{closure\,}_{\W\times\WW}\mK_3$ and a continuous map
$\theta^{-1}_{\w,x}\colon[-1,0]\to\mK_2$ such that:
$\theta^{-1}_{\w,x}(-1)=x_{-1}$; $\theta^{-1}_{\w,x}(0)=x$; and
\[
 \Pi(t,\ws,\theta^{-1}_{\w,x}(s))=
 (\w\pu(t+s),\theta^{-1}_{\w,x}(s+t))
 \text{ whenever } -1\le s\le-t\le 0\,.
\]
To this end, we take a sequence $((\w_m,x_m))$ in $\mK_3$ with limit $(\w,x)$.
For each $m\in\N$ we choose a backward orbit of $(\w_m, x_m)$ in $\mK_2$,
which we write as $\{(\w_m\pu s,\theta_{\w_m,x_m}(s))\,|$ $\;s\le 0\}$.
It is clear that $(\ws,\theta_{\w_m,x_m}(s))\in\mK_3$ for any $s\le 0$: its
backward orbit
is in fact provided by the same map. In addition, $\theta_{\w_m,x_m}(0)=x_m$, and
\[
 \Pi(t,\w_m\pu s,\theta_{\w_m,x_m}(s))=
 (\w_m\pu(t+s),\theta_{\w_m,x_m}(s+t))
 \text{ whenever } s\le-t\le 0\,.
\]
The compactness of $\mK_2$ provides a subsequence
$(\w_k,x_k)$ of $(\w_m,x_m)$ such that there exists
$\lim_{k\to\infty}(\w_k\pu(-1),\theta_{\w_k,x_k}(-1))$. We call
this limit $(\w\pu(-1),x_{-1})$ and note that
$(\w\pu(-1),x_{-1})\in\text{closure\,}_{\W\times\WW}\mK_3$.
We define
\[
 \theta^{-1}_{\w,x}\colon[-1,0]\to \mK_2\,,\quad s\mapsto
 u(1+s,\w\pu(-1),x_{-1})\,,
\]
which satisfies the required conditions: the positively $\Pi$-invariance of
$\mK_2$ ensures that it is well defined; it is obvious
that $\theta^{-1}_{\w,x}(-1)=x_{-1}$; in addition,
\[
\begin{split}
 \theta^{-1}_{\w,x}(0)&=u(1,\w\pu(-1),x_{-1})
 =\lim_{k\to\infty}u(1,\w_k\pu(-1),\theta_{\w_k,x_k}(-1))\\
 &=\lim_{k\to\infty}\theta_{\w_k,x_k}(0)=\lim_{k\to\infty}x_k=x
\end{split}
\]
(here we use Theorem~\ref{3.flujocontinuo}(vii) or (ix));
and finally, if $-1\le s\le -t\le 0$, then
\[
\begin{split}
 \Pi(t,\ws,\theta^{-1}_{\w,x}(s))
 &=(\w\pu(t+s),u(t,\ws,u(1+s,\w\pu(-1),x_{-1})))\\
 &=(\w\pu(t+s),u(1+t+s,\w\pu(-1),x_{-1}))
 =(\w\pu(t+s),\theta^{-1}_{\w,x}(t+s))\,.
\end{split}
\]
This completes the first step.
\par
Now we iterate the process in order to obtain a
sequence $((\w\pu(-j),x_{-j}))$ of points in $\text{closure\,}_{\W\times\WW}\mK_3$
and a sequence of continuous functions $(\theta_{\w,x}^{-j}\colon[-j,-j+1]\to \mK_2)$
with $\theta^{-j}_{\w,x}(-j+1)=\theta^{-j+1}_{\w,x}(-j+1)=x_{-j+1}$ such that
if $-j\le s\le -t-j+1\le -j+1$, then
\[
\begin{split}
 \Pi(t,(\w\pu(-j))\pu s,\theta^{-j}_{\w,x}(s))
 =((\w\pu(-j))\pu(t+s),\theta^{-j}_{\w,x}(t+s))\,.
\end{split}
\]
It is not hard to deduce from these facts that the continuous map
$\theta_{\w,x}\colon\R^-\to \mK_2$ obtained by concatenating the previous maps
is a backward extension of $(\w,x)$ in $\mK_2$.
This completes the proof of the compactness of
$\mK_3$.
\par
The last assertion of (iii) is obvious.
\end{proof}
As said before, in the rest of the section we
fix a set $\mK$ satisfying Hypotheses~\ref{3.hipos} (see also Remark~\ref{3.notapropK}).
Let us define $L\colon\mK\to\Lin(\WW,\R^{n}),\,(\w,\bar x)\mapsto L(\w,\bar x)$ by
\begin{equation}\label{3.defL}
\begin{split}
L(\w,\bar x)\phi&:=D_2F(\w,\bar x(0),\bar x(-\tau(\w,\bar x)))\phi(0)\\
&\,\quad+D_3F(\w,\bar x(0),\bar x(-\tau(\w,\bar x)))\phi(-\tau(\w,\bar x))\\
&\,\quad-D_3F(\w,\bar x(0),\bar x(-\tau(\w,\bar x)))
\dot{\bar x}(-\tau(\w,\bar x))\pu D_2\tau(\w,\bar x)\phi\,,
\end{split}
\end{equation}
and associate to~\eqref{3.eq} the family of {\em linear variational equations}
\begin{equation}\label{3.ecvar}
\dot{z}(t)=L(\Pi(t,\w,\bar x))z_t\,,\quad t\ge 0
\end{equation}
for $(\w,\bar x)\in\mK$. Let us summarize the strategy of the remaining
part of this section.
The solutions of this family of linear FDEs (of time-dependent delay type)
will allow us to define two semiflows on two different
bundles with base $(\mK,\Pi,\R^+)$. More precisely, on $\mK\times\WW$ and
on $\mK\times C$. Corollary 4.3 of \cite{mano1} states that the first one
is pseudo-continuous and the second one continuous.
The assumptions made on $\mK$ ensure that the construction made
in Remark~\ref{2.notalifting} applies to both semiflows, despite the lack
of global continuity of the first one. In particular, it makes sense
to talk about the upper Lyapunov exponents of these two linear
skew-product semiflows for which $(\mK,\Pi,\R^+)$ is the base. It is also proved
in \cite{mano1} (see Theorem~\ref{3.teorlin} below)
that the first semiflow is that usually called the {\em linearized semiflow of\/} $\Pi$.
This means that the corresponding upper Lyapunov exponent (which can be defined
despite the possible noncontinuity of the semiflow)
is that which responds to the classical concept. But it also
turns out that the second upper Lyapunov exponent is often \lq\lq easier to handle".
Theorem~\ref{3.igualexp} solves the disjunctive: it shows that in fact these
two quantities agree.
\par
We will now describe the two mentioned semiflows.
First, for each $(\w,\bar x)\in\mK$ and $v\in\WW$, we denote by
$z(t,\w,\bar x, v)$ the solution of \eqref{3.ecvar} with initial condition
$v$ (that is, with $z(s,\w,\bar x,v)=v(s)$ for each $s\in[-r, 0]$),
which is defined for all $t\ge -r$ and is linear with respect to $v$.
In addition, the map
$\mK\to\Lin(\WW, \R^{n}),\;(\w,\bar x)\mapsto L(\w,\bar x)$ is continuous.
These properties allow us to define a global linear skew-product semiflow
on the set $\mK\times\WW$ by
\begin{equation}\label{3.weak-skewlin}
\begin{array}{rcl}
\Pi_L\colon\R^+\!\times\mK\times\WW&\to&\mK\times\WW\\
(t,\w,\bar x, v)&\mapsto&(\Pi(t,\w,\bar x), w(t,\w,\bar x, v))\,,
\end{array}
\end{equation}
where $w(t,\w,\bar x,v)(s)=z(t+s,\w,\bar x,v)$ for all $s\in[-r,0]$ and $t\ge 0$.
As said before, Corollary 4.3 of \cite{mano1} proves that
this semiflow is pseudo-continuous. In particular, for all
$(t,\w,\bar x)\in\R^+\!\times \mK$, the linear map
\begin{equation}\label{3.defpiL}
 \pi_L(t,\w,\bar x)\colon\WW\to\WW,\quad v\mapsto w(t,\w,\bar x, v)
\end{equation}
is continuous.
\par
There is a strong relation between the semiflows $\Pi$ and $\Pi_L$, as the next
result shows. It is proved in Theorem 4.4 of \cite{mano1}, in turn based on Theorems 2
and 4 of \cite{hart1}.
\begin{teor}\label{3.teorlin}
Suppose that {\rm \hyperlink{3.H1}{H1}} and {\rm \hyperlink{3.H2}{H2}(1)} hold.
Let us fix $(\w,x)\in\mC_0$. If $t\in[0,\beta_{\w,x})$, then there exists
\begin{equation}\label{4.limW}
\begin{split}
 u_x(t,\w,x)v&=\lim_{{\ep\to 0}}\frac{u(t,\w,x+\ep v)-u(t,\w,x)}{\ep} \quad\text{in $\WW$}\\
\end{split}
\end{equation}
uniformly in $v\in\overline\mB_1:=\{v\in\WW\,|\; \n{v}_{\WW}=1\}$. In addition,
$u_x(t,\w,x)v=w(t,\w,x,v)$.
\end{teor}
The definition of the second semiflow is now given. For each $(\w,\bar x)\in \mK$ and
$v\in C$, let $\wit{z}(t,\w,\bar x, v)$ denote the solution of \eqref{3.ecvar} with
initial condition $v$. Corollary 4.3 of \cite{mano1}
proves that the solutions of \eqref{3.ecvar} induce a global continuous
linear skew-product semiflow on the set $\mK\times C$, defined by
\begin{equation}\label{3.tildeskew}
\begin{array}{rcl}
\wit\Pi_L\colon\R^+\!\times \mK\times C&\to&\mK\times C\\
(t,\w,\bar x, v)&\mapsto&(\Pi(t,\w,\bar x),\wit w(t,\w,\bar x, v))\,,
\end{array}
\end{equation}
where $\wit w(t,\w,\bar x, v)(s)=\wit z(t+s,\w,\bar x,v)$ for all $s\in[-r,0]$. We represent
\begin{equation}\label{3.defwpiL}
 \wit \pi_L(t,\w,\bar x)\colon C\to C\,,\quad v\mapsto \wit w(t,\w,\bar x, v)\,,
\end{equation}
which is a linear continuous map for all $(t,\w,\bar x)\in\R^+\!\times \mK$. Note that
\begin{equation}\label{3.proppiL}
 \wit \Pi_L(t,\w,\bar x,v)=(\Pi(t,\w,\bar x),\wit\pi_L(t,\w,\bar x)v)\,.
\end{equation}
\par
It is easy to deduce from the fact that $\wit z(t,\w,\bar x,v)$ solves~\eqref{3.ecvar},
from Hypotheses~\ref{3.hipos} and from
the expression of $L(\Pi(t,\w,\bar x))$ obtained from \eqref{3.defL} that,
if $(\w,\bar x)\in \mK$ and $v\in C$, then
$\wit w(t,\w,\bar x, v)\in\WW$ for $t\geq r$.
This means that the map $\wit\pi_L(t,\w,\bar x)$, defined on $C$ by \eqref{3.defwpiL},
takes values in $\WW$ for $t\ge r$. The next goal is to check that this map is
continuous when it is defined from $C$ to $\WW$. This property will be used in
the proof of Theorem~\ref{3.igualexp}.
\begin{prop}\label{3.propcompactflow}
Suppose that Hypotheses~\ref{3.hipos} hold. Given $(\w,\bar x)\in \mK$ and $T\geq r$, we define
\begin{equation}\label{3.defwhp}
\begin{array}{rcl}
\widehat\pi_L(T,\w,\bar x)\colon C&\to&\WW\\
v&\mapsto&\wit\pi_L(T,\w,\bar x)v\,,
\end{array}
\end{equation}
where $\wit\pi_L(T,\w,\bar x)$ is given by~\eqref{3.defwpiL}. Then
the map $\widehat\pi_L(T,\w,\bar x)$ is well-defined and continuous. If, in addition,
$T\geq 2r$, then the map $\widehat\pi_L(T,\w,\bar x)$ is compact.
Finally, if $T\ge 2r$, then the map $\pi_L(T,\w,\bar x)\colon\WW\to\WW$
given by~\eqref{3.defpiL} is compact.
\end{prop}
\begin{proof}
It has already been said that, if $v\in C$ and $T\geq r$, then
$\wit\pi_L(T,\w,\bar x)v\in\WW$, and hence the map $\widehat\pi_L(T,\w,\bar x)$
is well defined. It follows from this property and from \eqref{3.proppiL} that
$\widehat\pi_L(T,\w,\bar x)=\wit \pi_L(T-r, \Pi(r,\w,\bar x))\circ\widehat\pi_L(r,\w,\bar x)
=\pi_L(T-r, \Pi(r,\w,\bar x))\circ\widehat\pi_L(r,\w,\bar x)$, where $\pi_L$ is defined
by \eqref{3.defpiL}.
Thus, since $\pi_{L}(T-r, \Pi(r, \w,\bar x))\colon\WW\to\WW$
is continuous (see \eqref{3.defpiL}), in order to prove that
$\widehat\pi_L(T,\w,\bar x)\colon C\to\WW$ is continuous
it suffices to prove that $\widehat\pi_L(r,\w,\bar x)\colon C\to\WW$
is continuous.
\par
Take $v\in C$, so that $\widehat\pi_L(r,\w,\bar x)v=\wit w(r,\w,\bar x, v)=
\wit\pi_L(r,\w,\bar x)v$ and
\begin{equation}\label{3.derivz}
\begin{split}
 \dot{\wit z}(r+s,\w,\bar x,v)
 &
 =L(\Pi(r+s,\w,\bar x))\wit w(r+s,\w,\bar x,v)\\
 &=L(\Pi(r+s,\w,\bar x))(\wit\pi_L(r+s,\w,\bar x)v)
\end{split}
\end{equation}
for all $s\in[-r,0]$. Then,
\begin{equation}\label{3.hatpisubLnorma}
\begin{split}
 \n{\widehat\pi_L(r,\w,\bar x)v}_{\WW}&
 \le\n{\wit\pi_L(r,\w,\bar x)v}_C+\n{\dot{\wit w}(r,\w,\bar x,v)}_{L^\infty}\\
 &\le \wit C_r\,\n{v}_C+C_0\,\wit C_r\,\n{v}_C=\wit C_r\,(1+C_0)\,\n{v}_C\,,
\end{split}
\end{equation}
where
\begin{equation}\label{3.defC0Cr}
\begin{split}
 C_0&:=\!\sup_{(\w,\bar x)\in \mK}\!\n{L(\w,\bar x)}_{\Lin(C,\R^{n})}\,,\\
 \wit{C}_{r}&:=\!\!\sup_{\begin{smallmatrix}s\in[-r,0]\\(\w,\bar x)\in \mK\end{smallmatrix}}\!
 \n{\wit\pi_L(r+s,\w,\bar x)}_{\Lin(C, C)}\,.
\end{split}
\end{equation}
This proves the continuity.
\par
In order to prove the second assertion of the proposition, take $T\geq 2r$ and write, as before,
$\widehat\pi_L(T,\w,\bar x)=
\pi_L(T-2r, \Pi(2r,\w,\bar x))\circ\widehat\pi_L(2r,\w,\bar x)$. Since the map
$\pi_L(T-2r, \Pi(2r,\w,\bar x))\colon\WW\to\WW$ is continuous,
it suffices to prove that $\widehat\pi_L(2r,\w,\bar x)\colon C\to\WW$
is compact.
\par
Let us take a bounded sequence $(v_m)_{m\in\N}$ in C. It follows from \eqref{3.hatpisubLnorma}
that the sequence $(\n{\widehat\pi_L(r,\w,\bar x)v_m}_{\WW})_{m\in\N}$ is bounded.
Arzel\'{a}-Ascoli theorem provides a subsequence $(v_k)$ of $(v_m)$ such that
$(\wit\pi_L(r,\w,\bar x)v_k)$
converges to a function $\wit{v}\in C$. Hence, the sequence $(\widehat\pi_L(2r,\w,\bar x)v_k)
=\big((\widehat\pi_L(r,\Pi(r,\w,\bar x))\circ\wit\pi_L(r,\w,\bar x))v_k\big)$ converges to
$\widehat\pi_L(r,\Pi(r,\w,\bar x))\wit{v}\in\WW$. This shows the compactness of
$\widehat\pi_L(T,\w,\bar x)$ and hence of $\widehat\pi_L(T,\w,\bar x)$.
\par
The last assertion of the proposition is an immediate consequence of the previous
one and of the fact that any bounded sequence in $\WW$ determines a bounded
sequence in $C$. This completes the proof.
\end{proof}
For further purposes, we point out that
Proposition~\ref{3.propcompactflow} allows us to assert that
\begin{equation}\label{3.Cr2}
\widehat{C}_{r}:=\sup_{(\w,\bar x)\in \mK}\n{\widehat\pi_L(r,\w,\bar x)}_{\Lin(C,\WW)}
\end{equation}
is finite, where $\widehat\pi_L$ is defined by~\eqref{3.defwhp}. In fact, it follows from
\eqref{3.hatpisubLnorma} that $\widehat{C}_{r}\le\wit{C}_{r}\,(1+C_0)$, with $C_0$
and $\wit C_r$ given by~\eqref{3.defC0Cr}.
\par
Despite the possible lack of continuity of the semiflow $\Pi_L$
defined by \eqref{3.weak-skewlin},
Definition \ref{2.defLyap} provides two well-defined values, which we call
and denote
{\em upper Lyapunov exponent $\lambda^+_{s}(\w,\bar x)$ of $(\w,\bar x)\in\mK$ for
$(\mK,\Pi_L,\R^+)$} and
{\em upper Lyapunov exponent $\lambda_\mK$ of the semiflow $(\mK,\Pi_L,\R^+)$}.
We also denote by
$\wit{\lambda}^+_{s}(\w,\bar x)$ the upper Lyapunov exponent of $\wit\Pi_L$
(given by \eqref{3.tildeskew}) for $(\w,\bar x)\in \mK$; and by
$\wit{\lambda}_{\mK}$ the upper Lyapunov exponent of $(\mK,\wit\Pi_L,\R^+)$.
\begin{nota}
Despite the lack of classic continuity of the semiflow $\Pi_L$,
we can repeat the arguments of Theorem 4.2 of \cite{chle1} in order to check that,
if $(\w,\bar x)\in\mK$, then
\begin{equation}\label{2.lyap*}
 \lambda^+_{s}(\w,\bar x)=
 \limsup_{t\to\infty}\frac{1}{t}\:\ln\n{\pi_L(t,\w,\bar x)}_{\Lin(\WW,\WW)}\,.
\end{equation}
\end{nota}
The next theorem shows that the upper Lyapunov exponents of both semiflows coincide.
\begin{teor}\label{3.igualexp}
Suppose that Hypotheses~\ref{3.hipos} hold.
Let $\pi_L(t,\w,\bar x)$ and $\wit\pi_L(t,\w,\bar x)$ be defined by
\eqref{3.weak-skewlin} and \eqref{3.tildeskew} for $t\in\R$ and $(\w,\bar x)\in \mK$, and
let $C_0$ and $\widehat C_r$ be defined by~\eqref{3.defC0Cr} and
\eqref{3.Cr2}. And define $\lambda_{s}^+(\w,\bar x)$, $\lambda_\mK$,
$\wit{\lambda}_{s}^+(\w,\bar x)$ and $\wit\lambda_\mK$ as in the preceding paragraph.
The following statements hold:
\begin{itemize}
  \item[(i)] If $t\geq r$, then
  \[
  \qquad\n{\pi_L(t,\w,\bar x)}_{\Lin(\WW,\WW)}\le(1+C_0)
  \sup_{s\in[-r, 0]}\n{\wit\pi_L(t+s,\w,\bar x)}_{\Lin(C, C)}
  \]
  for every $(\w,\bar x)\in \mK$.
  \item[(ii)] $\lambda_{s}^+(\w,\bar x)\le \wit{\lambda}_{s}^+(\w,\bar x)$
  for every $(\w,\bar x)\in \mK$.
  \item[(iii)] If $t\geq r$, then
  \[
  \n{\wit\pi_L(t,\w,\bar x)}_{\Lin(C, C)}\le\widehat{C}_{r}\,
  \n{\pi_L(t-r, \Pi(r,\w,\bar x))}_{\Lin(\WW,\WW)}
  \]
  for every $(\w,\bar x)\in \mK$.
  \item[(iv)] $\wit{\lambda}_{s}^+(\w,\bar x)\le\lambda_{s}^+(\Pi(r,\w,\bar x))$
  for every $(\w,\bar x)\in \mK$.
  \item[(v)] $\lambda_{\mK}=\wit{\lambda}_{\mK}$.
\end{itemize}
\end{teor}
\begin{proof}
We take $t\geq r$, $(\w,\bar x)\in \mK$ and $v\in\WW\subset C$, and note that
\[
\n{\pi_L(t,\w,\bar x)v}_{C}=\n{\wit\pi_L(t,\w,\bar x)v}_{C}\le
\n{\wit\pi_L(t,\w,\bar x)}_{\Lin(C, C)}\,\n{v}_{C}\,.
\]
In addition, if $s\in[-r,0]$, we have
\[
 \dot z(t+s,\w,\bar x,v)=L(\Pi(t+s,\w,\bar x))(\wit\pi_L(t+s,\w,\bar x)v)\,:
\]
see~\eqref{3.derivz}, which is valid for $t$ instead of $r$, and recall that
$\wit z(t,\w,\bar x,v)=z(t,\w,\bar x,v)$, since $v\in\WW$.
Hence,
$|(d/ds)\,w(t,\w,\bar x,v)(s)|\le C_0\,\n{\wit\pi_L(t+s,\w,\bar x)}_{\Lin(C, C)}\,\n{v}_{C}$
for $s\in[-r,0]$. Thus, since $\n{v}_C\le\n{v}_{\WW}$,
\[
 \n{\pi_L(t,\w,\bar x)v}_{\WW}\le(1+C_0)
 \sup_{s\in[-r,0]}\n{\wit\pi_L(t+s,\w,\bar x)}_{\Lin(C, C)}\,\n{v}_{\WW}\,,
\]
and this proves (i). Now, (ii) is an easy consequence of the equalities
\eqref{2.lyap} and \eqref{2.lyap*}
for $\wit\lambda_s^+(\w,\bar x)$ and $\lambda_s^+(\w,\bar x)$.
\par
Now we take $t\geq r$, $(\w,\bar x)\in \mK$ and $v\in C$. Then,
\[
\begin{split}
\n{\wit\pi_L(t,\w,\bar x)v}_{C}&\le\n{\widehat\pi_L(t,\w,\bar x)v}_{\WW}\\
&=\n{\pi_L(t-r, \Pi(r,\w,\bar x))(\widehat\pi_L(r,\w,\bar x)v)}_{\WW}\\
&\le\n{\pi_L(t-r, \Pi(r,\w,\bar x))}_{\Lin(\WW,\WW)}\,
\n{\widehat\pi_L(r,\w,\bar x)v}_{\WW}\\
&\le\n{\pi_L(t-r,\Pi(r,\w,\bar x))}_{\Lin(\WW,\WW)}\,\widehat{C}_{r}\,\n{v}_{C}\,.
\end{split}
\]
This proves (iii). Property (iv) is an immediate consequence,
and (v) follows from (ii) and (iv).
\end{proof}
\begin{defi}\label{3.defly}
Suppose that Hypotheses~\ref{3.hipos} hold.
The {\em upper Lyapunov exponent of the set $\mK$
for the semiflow $(\mK,\Pi,\R^+)$}
is $\lambda_\mK=\wit\lambda_\mK$.
\end{defi}
\section{Exponential stability of invariant compact sets}\label{sec4}
The structure of Sections~\ref{sec4} and \ref{sec5} is similar: we establish conditions
characterizing the exponential stability of the positively invariant compact
subsets of the pseudo-continuous
semiflow $\Pi$ defined on $\W\times\WW$ by \eqref{3.defPi} in terms of the
corresponding upper Lyapunov exponents.
The hypotheses assumed in this section are more restrictive than those of the
next one, and they allow us to obtain stronger conclusions: compare the statements
of Theorems~\ref{4.teo} and~\ref{5.teo}.
Although part of the hypotheses are common, we write down now the whole list for
the reader's convenience. Recall that $(\W,\sigma,\R)$ is a continuous flow on a
compact metric space, with $\wt=\sigma(t,\w)$.
\begin{itemize}
\item[H1~~] \hypertarget{4.H1}
  $F\colon\W\times\R^{n}\times\R^{n}\to\R^{n}$ is continuous,
  and its partial derivatives with respect to the second and third arguments
  exist and are continuous on $\W\times\R^{n}\times\R^{n}$.
  In particular,
  the functions $D_iF\colon \W\times\R^{n}\times\R^{n}\to\Lin(\R^n,\R^n)$
  exist and are continuous for $i=2,3$.
\item[H2$^*$]
   \begin{itemize}
   \item[(1)] \hypertarget{4.H2*1}
   $\tau\colon\W\times C\to[0,r]$ is continuous and differentiable
   in the second argument, with $D_2\tau\colon\W\times C\to\Lin(C, \R)$ continuous.
   \item[(2)] \hypertarget{4.H2*2}
   $\tau$ is locally Lipschitz-continuous in this sense: for every
   bounded and closed subset $\mB\subset C$
   there exists a constant $L_1=L_1(\mB)>0$ such that
       \[
       \qquad |\tau(\w,x_1)-\tau(\w,x_2)|\le L_1\n{x_1-x_2}_{C}
       \]
        for all $\w\in\W$ and $x_1$, $x_2\in \mB$.
   \item[(3)] \hypertarget{4.H2*3}
   $D_2\tau$ is locally Lipschitz-continuous in the following sense:
   for every bounded and closed subset $\mB\subset C$ there exists a constant
   $L_2=L_2(\mB)>0$ such that
       \[
       \qquad \n{D_2\tau(\w,x_1)-D_2\tau(\w,x_2)}_{\Lin(C, \R)}
       \le L_2\n{x_1-x_2}_{C}
       \]
       for all $\w\in\W$ and $x_1$, $x_2\in \mB$.
   \end{itemize}
\end{itemize}
\par
Let $\Pi$ be defined by~\eqref{3.defPi} from the family \eqref{3.eq} of FDEs.
Throughout this section, we will work under
\begin{hipos}\label{4.hipos}
Conditions \hyperlink{4.H1}{H1} and \hyperlink{4.H2*1}{H2$^*$} hold,
and $\mK\subset\W\times\WW$ is
a positively $\Pi$-invariant compact set projecting over the whole base and
such that each one of its elements admits at least a backward extension in $\mK$.
\end{hipos}
Recall that the semiflow $(\mK,\Pi,\R^+)$ is global and continuous, and
that $\mK\subset\mC_0$: see Remark~\ref{3.notapropK}.
The goal of this section is to prove that the exponential stability of $\mK$ can be characterized
in terms of its upper Lyapunov exponent by the condition $\lambda_{\mK}<0$.
We will also show that the property of the exponential stability can be formulated
either in terms of the $\WW$-norm or of the $C$-norm.
These two results are stated in
the following theorem, whose proof requires three preliminary technical lemmas.
Recall Definition \ref{3.defly} of $\lambda_\mK$,
and that Theorem~\ref{3.igualexp} shows that
it is the upper Lyapunov exponent both form
the linearized semiflow given
by~\eqref{3.weak-skewlin} on $\mK\times\WW$ and by
\eqref{3.tildeskew} on $\mK\times C$; in fact, both definitions of $\lambda_\mK$
will be used in the proof.
It is interesting to remark that part of this result and the corresponding proof could be somehow
standard if the semiflow $\Pi$ were $C^1$ on an open neighborhood of $\mK$.
But the assumptions of this paper do not allow us to deduce this condition.
\begin{teor}\label{4.teo}
Suppose that Hypotheses~\ref{4.hipos} hold, and let $\lambda_\mK$ be
given by Definition \ref{3.defly}. The following statements are equivalent:
\begin{itemize}
  \item[(1)] $\lambda_{\mK}<0$.
  \item[(2)] There exist $\beta>0$, $k_1\geq1$, and $\delta_1>0$ such that, if
  $(\w,\bar x)\in \mK$ and $(\w,x)\in\W\times\WW$ satisfy
  $\n{x-\bar x}_{C}\le\delta_1$, then the function $y(t,\w,x)$ is defined for
  $t\in[-r,\infty)$ and
  \begin{equation}\label{4.desyteor}
      |y(t,\w, x)-y(t,\w,\bar x)|\le k_1e^{-\beta t}\,\n{x-\bar x}_{C}
      \quad\text{for all $t\geq -r$}\,,
  \end{equation}
  so that
  \begin{equation}\label{4.desuteor}
      \n{u(t,\w, x)-u(t,\w,\bar x)}_C\le k_1e^{\beta r}e^{-\beta t}\,
      \n{x-\bar x}_{C}\quad\text{for  all $t\geq 0$}\,.
  \end{equation}
  \item[(3)] The set $\mK$ is exponentially stable; i.e., there exist $\beta>0$, $k_2\geq1$,
  and $\delta_2>0$ such that, if $(\w,\bar x)\in \mK$ and $(\w,x)\in\W\times\WW$
  satisfy $\n{x-\bar x}_{\WW}<\delta_2$, then the function $u(t,\w, x)$
  is defined for  $t\in[0,\infty)$, and
      \[
      \quad \n{u(t,\w, x)-u(t,\w,\bar x)}_{\WW}\le k_2\,e^{-\beta t}\,
      \n{x-\bar x}_{\WW}\quad\text{for all $t\geq0$}\,.
      \]
\end{itemize}
In addition, if {\rm(i) holds}, we can take any $\beta\in(0,-\lambda_\mK)$ in {\rm(2)} and
{\rm(3)} (by changing the constants $\delta_1, k_1,\delta_2$ and $k_2$ if required).
\end{teor}
Before stating and proving the mentioned lemmas,
we fix some real parameters and a set which will play an important role
in what follows. We define
\[
\begin{split}
 r_0&:=1+\sup\{\n{\bar x}_{C}\,|\,(\w,\bar x)\in \mK\}\,,\\
 \mB_0&:=\{x\in C\,|\;\n{x}_C\le r_0\}\subset C\,,
\end{split}
\]
and represent by $L_1^0$ and $L_2^0$ the Lipschitz constants of the functions $\tau$ and
$D_2\tau$ on $\W\times \mB_0$, respectively provided by
conditions \hyperlink{4.H2*2}{H2$^*$(2)} and \hyperlink{4.H2*3}{H2$^*$(3)}.
We also denote
\[
\begin{split}
|F|_0&:=\sup\{|F(\w, h, k)|\;|\;\w\in\W\,,\,|h|\le r_0,\,|k|\le r_0\}\,,\\
\n{D_2F}_0&:=\sup\{\n{D_2F(\w, h, k)}_{\Lin(\R^{n}, \R^{n})}\,|\;
\w\in\W\,,\,|h|\le r_0,\,|k|\le r_0\}\,,\\
\n{D_3F}_0&:=\sup\{\n{D_3F(\w, h, k)}_{\Lin(\R^{n}, \R^{n})}\,|\;
\w\in\W,\,|h|\le r_0,\,|k|\le r_0\}\,,\\
\n{D_2\tau}_0&:=\sup\{\n{D_2\tau(\w,\bar x+x)}_{\Lin(C, \R)}\,|\;
(\w,\bar x)\in \mK,\,\n{x}_{C}\le1\}\,.
\end{split}
\]
It follows from \hyperlink{4.H2*3}{H2$^*$(3)} that
$\n{D_2\tau}_0\le L_2^0+\sup\{\n{D_2\tau(\w,\bar x)}_{\Lin(C,\R)}\,|\,
(\w,\bar x)\in \mK\}$, so that it is finite. Condition \hyperlink{4.H1}{H1}
ensures the same property for
$|F|_0$, $\n{D_2F}_0$ and $\n{D_3F}_0$. We assume without restriction
that the six constants $L_1^0$, $L_2^0$, $|F|_0$, $\n{D_2F}_0$, $\n{D_3F}_0$ and
$\n{D_2\tau}_0$ are strictly positive.
\par
Recall the notations $y(t,\w,x)$ and $u(t,\w,x)$ established in Section~\ref{sec3}.
Recall also that they are defined for $t\in[-r,\beta_{\w,x})$ and $t\in[0,\beta_{\w,x})$,
respectively. And note that $\beta_{\w,\bar x}=\infty$ if $(\w,\bar x)\in \mK$.
In the proofs of the next three lemmas, and in that of Theorem~\ref{4.teo}, we will be working
under Hypotheses~\ref{4.hipos}, and with two previously fixed points
$(\w,x)\in\W\times\WW$ and $(\w,\bar x)\in \mK$.
Therefore, the functions $y(t,\w,x)$ and $y(t,\w,\bar x)$ are both
defined in $[-r,\beta_{\w,x})$
and the functions $u(t,\w,x)$ and $u(t,\w,\bar x)$ are both
defined on $[0,\beta_{\w,x})$.
To simplify the notation, we will represent
\begin{equation}\label{4.notacion}
\begin{array}{lll}
 y(t):=y(t,\w,x)&\text{and\;}\quad\bar y(t):=
 y(t,\w,\bar x)&\text{\;\;for $\;t\in[-r,\beta_{\w,x})\,,$}\\
 \wit y(t):=y(t)-\bar y(t)&&
 \text{\;\;for $\;t\in[-r,\beta_{\w,x})\,,$}\\
 u(t):=u(t,\w,x)&\text{and\;}\quad\bar u(t):=u(t,\w,\bar x)&
 \text{\;\;for $\;t\in[0,\beta_{\w,x})\,.$}
\end{array}
\end{equation}
We will also be working under the assumption
\begin{equation}\label{4.cota1}
 \n{u(t)-\bar u(t)}_C\le 1\quad\text{\;\;for all $t\in[0,T]$}
\end{equation}
where $T$ is a fixed time in $[0,\beta_{\w,x})$. This inequality
together with the fact that $(\wt,\bar u(t))$ belongs to $\mK$ ensures that
\begin{equation}\label{4.cotayt}
\!\!\!\!\!\!\!\begin{array}{ll}
|y(t+s)|\le|\bar y(t+s)|+\n{u(t)-\bar u(t)}_C\le r_0 &\text{for $t\in[0,T]$ and
$s\in[-r,0]$}\,,\\[.15cm]
|\bar y(t+s)|< r_0 &\text{for $t\in[0,T]$ and
$s\in[-r,0]$}\,.
\end{array}
\end{equation}
\begin{lema}\label{4.lem1}
Suppose that Hypotheses~\ref{4.hipos} hold. Then, for every $\ep>0$ there exists
$\delta_3=\delta_3(\ep)\in(0,1]$ such that, if $(\w,\bar x)\in \mK$,
$(\w,x)\in\W\times\WW$, $T\in[0,\beta_{\w,x})$, and
$\n{u(t)-\bar{u}(t)}_{C}\le\delta_3$ for every
$t\in[0, T]$, then
\[
\begin{split}
&\big|\bar{y}(t-\tau(\wt,u(t))) - \bar{y}(t-\tau(\wt,\bar{u}(t)))\\
&\qquad\qquad+\dot{\bar{y}}(t-\tau(\wt,\bar{u}(t)))\pu D_2\tau(\wt,\bar{u}(t))(u(t) - \bar{u}(t))\big|\le\ep\,\n{u(t) - \bar{u}(t)}_{C}
\end{split}
\]
for every $t\in[0, T]$. The notation~\eqref{4.notacion} is used in this statement.
\end{lema}
\begin{proof}
Let us fix $(\w,\bar x)\in \mK$, $(\w,x)\in\W\times\WW$ and $T\in[0,\beta_{\w,x})$.
The notation~\eqref{4.notacion} will be used in the proof. In what follows we
will assume (without loss of generality) that~\eqref{4.cota1} holds,
and hence that inequalities~\eqref{4.cotayt} are valid.
\par
Since $\mK\subset\mC_0$ (see Remark \ref{3.notapropK}), we have
$\bar{y}\in C^1([-r, T], \R^{n})$.
For each $t\in[0,T]$, we define $p_{\,t}\colon[0,1]\to\R^{n}$ by
\[
\begin{split}
 p_{\,t}(s)&:=\bar{y}(t-\tau(\wt,\bar{u}(t)+s(u(t)-\bar{u}(t))))\\
 &\,\quad+s\,\dot{\bar{y}}(t-\tau(\wt,\bar{u}(t)))\pu D_2\tau(\wt,\bar{u}(t))
 (u(t) - \bar{u}(t))\,.
\end{split}
\]
Thus,
\[
\begin{split}
 p_{\,t}(1)-p_{\,t}(0)&=\bar{y}(t-\tau(\wt,u(t))) - \bar{y}(t-\tau(\wt,\bar{u}(t)))\\
 &\quad+ \dot{\bar y}(t-\tau(\wt,\bar{u}(t)))\pu D_2\tau(\wt,\bar{u}(t))(u(t) - \bar{u}(t))\,,
\end{split}
\]
so that the assertion of the lemma is equivalent to the property
\begin{equation}\label{4.cotaep}
 \Frac{|p_{\,t}(1)-p_{\,t}(0)|}{\n{u(t) - \bar{u}(t)}_{C}}\le\ep
 \qquad\text{for all $t\in[0,T]\;$ with $u(t)\ne \bar u(t)$}\,.
\end{equation}
The chain rule implies that $p_{\,t}$ is continuously differentiable, and that
\[
\begin{split}
\dot{p_{\,t}}(s)&=\big\{\!-\dot{\bar{y}}(t-\tau(\wt,\bar{u}(t)+s\,
(u(t)-\bar{u}(t))))\pu D_2\tau(\wt,\bar{u}(t)+s(u(t)-\bar{u}(t)) )\\
&\quad+\dot{\bar{y}}(t-\tau(\wt,\bar{u}(t)))\pu D_2\tau(\wt,\bar{u}(t))\big\}(u(t) -
\bar{u}(t))\,.
\end{split}
\]
Since $p_{\,t}(1)-p_{\,t}(0)=\int_0^1 \dot{p}_{\,t}(s)\,ds$, we can take
$s\in[0,1]$ with
$|p_{\,t}(1)-p_{\,t}(0)|\le |\dot p_{\,t}(s)|$.
Having in mind the definitions of $\n{D_2\tau}_0$ and $|F|_0$,
and the second bound in~\eqref{4.cotayt}, we have, if $u(t)\ne\bar u(t)$,
\[
\begin{split}
 &\Frac{|p_{\,t}(1)-p_{\,t}(0)|}{\n{u(t)-\bar{u}(t)}_{C}}\le
 \Frac{|\dot{p_{\,t}}(s)|}{\n{u(t)-\bar{u}(t)}_{C}}\\
 &\quad\le|\dot{\bar{y}}(t-\tau(\wt,\bar{u}(t)))-
 \dot{\bar y}(t-\tau(\wt,\bar{u}(t)+s(u(t)-\bar{u}(t))))|\cdot\\
 &\qquad\qquad\qquad\qquad\qquad\qquad\qquad\qquad\quad
 \cdot\n{D_2\tau(\wt,\bar{u}(t)+s(u(t)-\bar{u}(t)))}_{\Lin(C, \R)}\\
 &\quad\;\;+|\dot{\bar y}(t-\tau(\wt,\bar{u}(t)))|\,\n{D_2\tau(\wt,\bar{u}(t)) -
 D_2\tau(\wt,\bar{u}(t)+s(u(t)-\bar{u}(t)))}_{\Lin(C, \R)}\\
 &\quad\le\n{D_2\tau}_0\,|\dot{\bar{y}}(t-\tau(\wt,\bar{u}(t))) -
 \dot{\bar y}(t-\tau(\wt,\bar{u}(t)+s(u(t)-\bar{u}(t))))|\\
 &\quad\quad+|F|_0\,\n{D_2\tau(\wt,\bar{u}(t)) -
 D_2\tau(\wt,\bar{u}(t)+s(u(t)-\bar{u}(t)))}_{\Lin(C, \R)}\,.
\end{split}
\]
The restricted map $\Pi\colon[0,1]\times \mK\to\mK$,
$(t,\w,\bar x)\mapsto(\wt,u(t,\w,\bar x))$ is uniformly continuous.
Therefore, given our $\ep>0$, there exists $\rho_\ep\in(0,1]$ such that,
if $0\le s\le\rho_\ep$, then
$\n{u(s,\w_0,\bar x_0)-\bar x_0}_{\WW}=\n{u(s,\w_0,\bar x_0)-
u(0,\w_0,\bar x_0)}_{\WW}\le\ep/(2\,\n{D_2\tau}_0)$
for every $(\w_0,\bar x_0)\in \mK$. Now we take the point $(\w^*,\bar x^*)\in\mK$
such that $(\w,\bar x)=\Pi(r,\w^*,\bar x^*)$. Then, if $-r\le t_1<t_2\le t_1+\rho_\ep$,
we obtain
\[
\begin{split}
 |\dot{\bar{y}}(t_2)-\dot{\bar{y}}(t_1)|&\le\n{u(t_2+r,\w^*,\bar x^*)-
 u(t_1+r,\w^*,\bar x^*)}_{\WW}\\[.1cm]
 &=\n{u(t_2-t_1,\w^*\pu(t_1+r),u(t_1+r,\w^*,\bar x^*))-
 u(t_1+r,\w^*,\bar x^*)}_{\WW}\\
 &\le\Frac{\ep}{2\,\n{D_2\tau}_0}\:.
\end{split}
\]
In addition to~\eqref{4.cota1}, we assume that
\begin{equation}\label{4.cota2}
 \n{u(t)-\bar u(t)}_C\le \frac{\rho_\ep}{L_1^0}\qquad\text{for all $t\in[0,T]$}\,,
\end{equation}
so that
\[
 |\tau(\wt,\bar{u}(t))-\tau(\wt,\bar{u}(t)+s(u(t)-\bar{u}(t)))|\le
 L_1^0\,s\,\n{u(t)-\bar{u}(t)}_C\le\rho_\ep\,,
\]
and consequently, for all $t\in[0,T]$,
\[
 |\dot{\bar{y}}(t-\tau(\wt,\bar{u}(t))) - \dot{\bar{y}}
 (t-\tau(\wt,\bar{u}(t)+s(u(t)-\bar{u}(t))))|\le \frac{\ep}{2\,\n{D_2\tau}_0}\;.
\]
Finally, in addition to~\eqref{4.cota1} and~\eqref{4.cota2}, we assume that
\begin{equation}\label{4.cota3}
 \n{u(t)-\bar u(t)}_C\le \frac{\ep}{2\,L_2^0\,|F|_0}\qquad\text{for all $t\in[0,T]$}\,,
\end{equation}
so that, for all $t\in[0,T]$,
\[
 \n{D_2\tau(\wt,\bar{u}(t)) - D_2\tau(\wt,\bar{u}(t)+s(u(t)-\bar{u}(t)))}_{\Lin(C, \R)}
 \!\le\!L_2^0\,s\,\n{u(t)-\bar{u}(t)}_{C}\!\le\!\Frac{\ep}{2\,|F|_0}\:.
\]
Altogether, these properties show that~\eqref{4.cotaep} holds
if~\eqref{4.cota1}, \eqref{4.cota2} and \eqref{4.cota3} hold, for which it suffices to take
$\delta_3:=\min\big(1,\,\rho_\ep/L_1^0,\,\ep/(2L_2^0|F|_0)\big)$ and
$\n{u(t)-\bar u(t)}_C\le \delta_3$ for all $t\in[0,T]$. This is the value of $\delta_3$
appearing in the statement.
\end{proof}
\begin{lema}\label{4.lem2}
Suppose that Hypotheses~\ref{4.hipos} hold. Then, for every $\ep>0$ there exists
$\delta_4=\delta_4(\ep)\in(0,1]$ such that, if $(\w,\bar x)\in \mK$,
$(\w,x)\in\W\times\WW$, $T\in[0,\beta_{\w,x})$, and
$\n{u(t)-\bar{u}(t)}_{C}\le\delta_4$ for every
$t\in[0, T]$, then
\[
 |\wit y(t-\tau(\wt,u(t)))-\wit y(t-\tau(\wt,\bar{u}(t)))|\le\left\{\!
 \begin{array}{ll}
 2\,\n{u(t)-\bar{u}(t)}_{C}&\text{if}\;\;\;t\in[0,r]\,,\\[0.1cm]
 \ep\,\n{u(t)-\bar{u}(t)}_{C}&\text{if}\;\;\;t\in[r,T]\,.
 \end{array}
 \right.
\]
The notation~\eqref{4.notacion} is used in this statement.
\end{lema}
\begin{proof}
Let us fix $(\w,\bar x)\in \mK$, $(\w,x)\in\W\times\WW$ and
$T\in[0,\beta_{\w,x})$. The notation~\eqref{4.notacion} will be used in the proof.
And we will assume that~\eqref{4.cota1} holds, so that also~\eqref{4.cotayt} holds.
\par
The inequality is almost immediate for $t\in[0,r]$, so we must just consider the
case $t\in[r, T]$. Note that
\begin{equation}\label{4.pr}
\begin{split}
 &|\wit y(t-\tau(\wt,u(t)))-\wit y(t-\tau(\wt,\bar{u}(t)))|\\
 &\qquad\qquad\le\Big(\,\max_{s\in[t-r,t]}|\dot{\wit y}(s)|\,\Big)\,
 |\tau(\wt,u(t))-\tau(\wt,\bar{u}(t))|\,.
\end{split}
\end{equation}
We take $s\in[t-r,t]\subseteq[0,T]$ and use \eqref{3.dery} to calculate
$\dot{\wit y}(s)=\dot{y}(s)-\dot{\bar{y}}(s)$. Since
\begin{equation}\label{4.int}
 F(\w,y_1,y_2)-F(\w,\bar y_1,\bar y_2)=\!\int_0^1\!\Frac{d}{d\nu}\,F(\w,\nu y_1+(1-\nu)\bar y_1,
 \nu y_2+(1-\nu)\bar y_2)\,d\nu
\end{equation}
and $|\nu z_1+(1-\nu)z_2|\le r_0$ if $|z_1|\le r_0$ and $|z_2|\le r_0$, we can apply
\eqref{4.cotayt} and the definitions of $\n{D_2F}_0$ and $\n{D_3F}_0$ in order to obtain
\[
\begin{split}
&|\dot{\wit y}(s)|=|\dot{y}(s)-\dot{\bar{y}}(s)|\\
&\quad=\big|F(\ws, y(s), y(s-\tau(\ws, u(s))))-
F(\ws, \bar{y}(s), \bar{y}(s-\tau(\ws, \bar{u}(s))))\big|\\
&\quad\le\n{D_2F}_0\,|y(s)-\bar{y}(s)|+\n{D_3F}_0\,
\big|y(s-\tau(\ws, u(s)))-\bar{y}(s-\tau(\ws, \bar{u}(s)))\big|\\
&\quad\le\n{D_2F}_0\,\big|y(s)-\bar{y}(s)\big|+\n{D_3F}_0\,
\Big(\big|y(s-\tau(\ws, u(s)))-\bar{y}(s-\tau(\ws, u(s)))\big|\\
&\quad\quad+\big|\bar{y}(s-\tau(\ws, u(s)))-\bar{y}(s-\tau(\ws, \bar{u}(s)))\big|\Big).
\end{split}
\]
Note that, due to~\eqref{4.cotayt}, $u(t)$ and $\bar u(t)$ belong to $\mB_0$.
Hence, by the definition of $L_1^0$,
\begin{equation}\label{4.cota5}
|\tau(\wt,u(t))-\tau(\wt,\bar{u}(t))|\le L_1^0\,\n{u(t)-\bar{u}(t)}_{C}\,.
\end{equation}
Therefore, using again~\eqref{3.dery} and \eqref{4.cotayt}, we obtain
\begin{equation}\label{4.paluego}
|\bar{y}(s-\tau(\ws, u(s)))-\bar{y}(s-\tau(\ws, \bar{u}(s)))|
\le |F|_0L_1^0\,\n{u(s)-\bar{u}(s)}_{C}\,.
\end{equation}
Altogether, we have
\[
\begin{split}
|\dot{\wit y}(s)|&
\le\n{D_2F}_0\,\n{u(s)-\bar{u}(s)}_{C}\\
&\quad +\n{D_3F}_0\big(\n{u(s)-\bar{u}(s)}_{C}+ |F|_0L_1^0\,\n{u(s)-\bar{u}(s)}_{C}\big)\\
&=\big(\n{D_2F}_0+\n{D_3F}_0\,(1+|F|_0L_1^0)\big)\,\n{u(s)-\bar{u}(s)}_{C}.
\end{split}
\]
This inequality together with those given by \eqref{4.cota5} and~\eqref{4.pr} prove the lemma for
$\delta_4:=\min\big(1,\ep/\big(\n{D_2F}_0+\n{D_3F}_0(1+|F|_0L_1^0)\,L_1^0\,\big)\big)$.
\end{proof}
\begin{lema}\label{4.lem3}
Suppose that Hypotheses~\ref{4.hipos} hold, and define $g\colon\mK\times\WW\to\R^{n}$ by
\[
 g(\w,\bar x, x):=F(\w,x(0), x(-\tau(\w,x)))-F(\w,\bar x(0),
 \bar x(-\tau(\w,\bar x)))-L(\w,\bar x)(x-\bar x)\,,
\]
where $L(\w,\bar x)$ is given by~\eqref{3.defL}.
Then, for every $\ep\in(0,(4/3)\n{D_3F}_0]$ there exists $\delta_5=\delta_5(\ep)\in(0,1]$
such that, if $(\w,\bar x)\in \mK$, $(\w,x)\in\W\times\WW$, $T\in[0,\beta_{\w,x})$,
and $\n{u(t)-\bar{u}(t)}_{C}\le\delta_5$ for every $t\in[0, T]$, then
\[
|g(\wt,\bar{u}(t), u(t))|\le\left\{
 \begin{array}{ll}
 3\,\n{D_3F}_0\,\n{u(t)-\bar{u}(t)}_{C}&\text{if}\;\;\;t\in[0,r]\,,\\[0.1cm]
 \ep\,\n{u(t)-\bar{u}(t)}_{C}&\text{if}\;\;\;t\in[r,T]\,.
 \end{array}
 \right.
\]
The notation~\eqref{4.notacion} is used in this statement.
\end{lema}
\begin{proof}
Let us fix $(\w,\bar x)\in \mK$, $(\w,x)\in\W\times\WW$ and $T\in[0,\beta_{\w,x})$.
The notation~\eqref{4.notacion} will be used in the proof. In what follows we
will assume that~\eqref{4.cota1} holds, and hence also~\eqref{4.cotayt} is valid.
\par
We write
$w(t):=y(t-\tau(\wt,u(t)))$ and $\bar w(t):=\bar y(t-\tau(\wt,\bar u(t)))$
for $t\in[0,T]$.~The
definitions of $g(\wt,\bar u(t),u(t))$ and $L(\wt,\bar u(t))$ (see \eqref{3.defL})
together with \eqref{4.int}
yield
\[
\begin{split}
&|g(\wt,\bar{u}(t), u(t))|\\
&\quad=\big|F(\wt,y(t), w(t))-F(\wt,\bar y(t), \bar w(t))-
D_2F(\wt,\bar y(t), \bar w(t))(y(t)-\bar y(t))\\
&\quad\quad-D_3F(\wt,\bar y(t),\bar w(t))
(y(t-\tau(\wt,\bar u(t)))-\bar y(t-\tau(\wt,\bar u(t))))\\
&\quad\quad+D_3F(\wt,\bar y(t),\bar w(t))\,\dot{\bar{y}}(t-\tau(\wt,\bar{u}(t)))
\cdot D_2\tau(\wt,\bar{u}(t))(u(t)-\bar{u}(t))\big|\\
&\quad\le\sup_{\nu\in[0,1]}\big\|D_2F(\wt,\bar y(t)+\nu(y(t)-\bar y(t)),
\bar w(t)+\nu(w(t)-\bar w(t)))\\[-.2cm]
&\quad\;\qquad\qquad-D_2F(\wt,\bar y(t), \bar w(t))\big\|_{\Lin(\R^{n}, \R^{n})}
\,|y(t)-\bar y(t)|\\
&\quad\quad+\sup_{\nu\in[0,1]}\big\|D_3F(\wt,\bar y(t)+\nu(y(t)-\bar y(t)),
\bar w(t)+\nu(w(t)-\bar w(t)))\\[-.1cm]
&\quad\quad\qquad\qquad-D_3F(\wt,\bar y(t),\bar w(t))\big\|_{\Lin(\R^{n}, \R^{n})}
\,|w(t)-\bar w(t)|\\
&\quad\quad+\big\|D_3F(\wt,\bar y(t),\bar w(t))\big\|_{\Lin(\R^{n}, \R^{n})}
\big|y(t-\tau(\wt,u(t)))-y(t-\tau(\wt,\bar{u}(t)))\\
&\qquad\qquad\qquad\qquad\qquad\qquad\quad
+\dot{\bar{y}}(t-\tau(\wt,\bar{u}(t)))\cdot D_2\tau(\wt,\bar{u}(t))(u(t)-\bar{u}(t))\big|\,.
\end{split}
\]
Let us fix $\ep\in(0,(4/3)\n{D_3F}_0]$. The last sum has three terms.
Each of the two first ones is bounded by $(\ep/4)\,\n{u(t)-\bar u(t)}_C$ for $t\in[0,T]$.
In order to check this assertion, note that
\[
\begin{split}
 |y(t)-\bar y(t)|&\le \n{u(t)-\bar u(t)}_C \qquad\text{\;\;for $t\ge 0$}\,,\\
 |w(t)-\bar w(t)|&\le (1+|F|_0L_1^0)\,\n{u(t)-\bar u(t)}_C \qquad\text{\;\;for $t\ge 0$}\,.\\
\end{split}
\]
The first inequality is obvious. To prove the second one, use \eqref{4.paluego} to check that
\[
\begin{split}
 |w(t)-\bar w(t)|&\le |y(t-\tau(\wt,u(t)))-\bar y(t-\tau(\wt,u(t)))|\\
 &\quad+|\bar y(t-\tau(\wt,u(t)))-\bar y(t-\tau(\wt,u(t)))|\\
 &\le \n{u(t)-\bar u(t)}_C+|F|_0L_1^0\,\n{u(t)-\bar u(t)}_C\,.
\end{split}
\]
In addition, since Hypotheses~\ref{4.hipos} hold, there exists $\rho\in(0,1]$
such that, if $|\bar h|\le r_0$ and $|\bar k|\le r_0$ (as is the case of
$\bar y(t)$ and $\bar w(t)$, according to \eqref{4.cotayt}), and if
$|h-\bar{h}|<\rho$ (as it happens with $|\nu(y(t)-\bar y(t))|$ for
all $\nu\in[0,1]$ if $\n{u(t)-\bar u(t)}_C\le\rho$) and
$|k-\bar{k}|<\rho\,(1+|F|_0L_1^0)$ (as it happens with $|\nu(w(t)-\bar w(t))|$ for
all $\nu\in[0,1]$ if $\n{u(t)-\bar u(t)}_C\le\rho$), then, for every $\w\in\W$, it is
$\n{D_2F(\w, h, k)-D_2F(\w, \bar{h}, \bar{k})}_{\Lin(\R^{n}, \R^{n})}<
\ep/4$ and $\n{D_3F(\w, h, k)-D_3F(\w, \bar{h}, \bar{k})}_{\Lin(\R^{n}, \R^{n})}<
\ep/(4+4\,|F|_0L_1^0)\,$.
It follows easily that the assertion concerning the bound of the two first terms
is true if $\n{u(t)-\bar u(t)}_C<\rho$ for all $t\in[0,T]$.
\par
To bound the last term, note that
$\|D_3F(\wt,\bar y(t),\bar w(t))\big\|_{\Lin(\R^{n}, \R^{n})}\le\n{D_3F}_0$:
use \eqref{4.cotayt} and the definition of $\n{D_3F}_0$. In addition,
\[
\begin{split}
&\big|y(t-\tau(\wt,u(t))) - y(t-\tau(\wt,\bar{u}(t)))\\
&\qquad\qquad\qquad\qquad
+\dot{\bar{y}}(t-\tau(\wt,\bar{u}(t)))\pu D_2\tau(\wt,\bar{u}(t))(u(t) - \bar{u}(t))\big|\\
&\qquad\le \big|\bar{y}(t-\tau(\wt,u(t))) - \bar{y}(t-\tau(\wt,\bar{u}(t)))\\
&\qquad\qquad\qquad\qquad\qquad\quad\;
+\dot{\bar{y}}(t-\tau(\wt,\bar{u}(t)))\pu D_2\tau(\wt,\bar{u}(t))(u(t) - \bar{u}(t))\big|\\
&\qquad\quad+|\wit y(t-\tau(\wt,u(t)))-\wit y(t-\tau(\wt,\bar{u}(t)))|\,.
\end{split}
\]
Lemma \ref{4.lem1} provides $\delta_3\in(0,1]$ (irrespective of $\w,\,x,\,\bar x$
and $T$) such that, if $\n{u(t)-\bar u(t)}_C\le \delta_3$ for all $t\in[0,T]$,
then  the first term of the last sum is bounded by
$\big(\ep/(4\,\n{D_3F}_0)\big)\,\n{u(t)-\bar u(t)}_C$
for all $t\in[0,T]$.
In addition, Lemma \ref{4.lem2} ensures the existence of $\delta_4\in(0,1]$
(also irrespective of $\w,\,x,\,\bar x$ and $T$) such that,
if $\n{u(t)-\bar{u}(t)}_{C}\le\delta_4$ for all $t\in[0,T]$, then
\[
 |\wit y(t-\tau(\wt,u(t)))-\wit y(t-\tau(\wt,\bar{u}(t)))|\le\left\{\!\begin{array}{ll}
 2\,\n{u(t)-\bar{u}(t)}_{C}&\!\!\text{if}\;\,t\in[0,r]\\[.1cm]
 \Frac{\ep}{4\,\n{D_3F}_0}\:\n{u(t)-\bar{u}(t)}_{C}&\!\!\text{if}\;\,t\in[r,T]
 \end{array}
 \right.
\]
Altogether, if we take $\delta_5:=\min\,(\rho,\,\delta_3,\,\delta_4)\le 1$
and assume that $\n{u(t)-\bar u(t)}_C\le \delta_5$ for all $t\in[0,T]$, we have
\[
|g(\wt,\bar{u}(t), u(t))|\le\left\{
 \begin{array}{ll}
 \!\!\left(\Frac{3\,\ep}{4}+2\,\n{D_3F}_0\right)
 \n{u(t)-\bar{u}(t)}_{C}&\text{if}\;\;\;t\in[0,r]\,,\\[0.1cm]
 \ep\,\n{u(t)-\bar{u}(t)}_{C}&\text{if}\;\;\;t\in[r,T]\,;
 \end{array}
 \right.
\]
and, since $3\ep/4\le \n{D_3F}_0$, this proves the statement of the lemma.
\end{proof}
\par
We can finally prove the main theorem of this section.
\medskip\par\noindent
{\em Proof of Theorem~\ref{4.teo}}.
(1)$\Rightarrow$(2)  We consider the linear system
\begin{equation}\label{4.linhomeq}
\dot{y}(t)=L(\Pi(t,\w,\bar x))y_{t}\,,\quad t\geq0\,,
\end{equation}
where $L$ is defined by~\eqref{3.defL}.
Let $U(t,\w,\bar x)$ be the fundamental solution of \eqref{4.linhomeq} in the terms
given in Chapter 1 of \cite{havl}; i.e.,
for each $(\w,\bar x)\in \mK$ the $n\times n$ matrix-valued map
$t\to U(t,\w,\bar x)$ is a solution of $\dot U(t)=L(\Pi(t,\w,\bar x))U_t$
for $t\geq0$, and it satisfies
\begin{eqnarray*}
U(t,\w,\bar x)=\left\{\!
                    \begin{array}{ll}
                      I_n & \text{for $t=0$\,,}\\
                      0_n & \text{for $t\in[-r, 0)$}
                    \end{array}
                  \right.
\end{eqnarray*}
for all $(\w,\bar x)\in \mK$. Here $I_n$ and $0_n$ are the $n\times n$ identity and
zero matrices.
\par
We assume that $\lambda_{\mK}<0$, fix any $\beta\in(0,-\lambda_{\mK})$, and choose
$\alpha$ with $\beta<\alpha<-\lambda_{\mK}$. Theorem~\ref{3.igualexp} together with the
expression~\eqref{3.defwpiL} of the flow on $\mK\times C$ and relation
\eqref{2.lyap2} ensures the existence of a constant $k_0\ge 1$
such that $\n{\wit w(t,\w,\bar x, v)}_{C}=\n{\wit\pi_L(t,\w,x)\,v}_C\le
k_0\,e^{-\alpha t}\,\n{v}_{C}$ and $|U(t,\w,\bar x)\,c|\le k_0\,e^{-\alpha\,t}\,|c|$
for every $(\w,\bar x)\in \mK$, $v\in C$, $c\in\R^{n}$ and $t\geq0$.
\par
We fix $\ep>0$ small enough to apply Lemma \ref{4.lem3} and satisfying the additional
bound $0<\ep\,k_0\,e^{\beta r}/(\alpha-\beta)<1/2$. Let $\delta_5=\delta_5(\ep)$ be
the real number provided by  Lemma \ref{4.lem3}.
Recall that the functions $y(t)$ and $\bar y(t)$ are defined on $[-r,\beta_{\w,x})$.
We take $(\w,\bar x)\in \mK$ and $(\w,x)\in\W\times\WW$,
and use the notation~\eqref{4.notacion} from now on. It is easy to check that
$\wit y(t)=y(t)-\bar y(t)$ (which satisfies $\wit y_t=u(t)-\bar{u}(t)$) is a solution of the FDE
\begin{equation}\label{4.ecteor}
 \dot{\wit y}(t)=L(\Pi(t,\w,\bar x))\wit y_t+g(\wt,\bar{u}(t), u(t))
\end{equation}
for every $t\in[0,\beta_{\w,x})$, where $g$ is defined in the statement of Lemma \ref{4.lem3}.
We apply an adapted version of the variation of constants formula (see Section 2 of Chapter 6
of~\cite{havl}) in order to represent $\wit y(t)$ as
\[
 \wit y(t)\!=\!\left\{\!\!\begin{array}{ll}
 x(t)-\bar x(t)&\!\text{if $t\in[-r,0)$\,,}\\
 \wit z(t,\w,\bar x, x-\bar x)+\int_0^{t}U(t-s,\ws, \bar{u}(s))\,
 g(\ws, \bar{u}(s), u(s))\,ds&
 \!\text{if $t\ge 0$}\,,\end{array}\right.
\]
where $\wit{z}(t,\w,\bar x, v)$ is the solution of \eqref{3.ecvar} with
initial condition $v\in C$.
\par
We begin by considering the case $t\in[0,r]$. Let us assume that $\n{x-\bar x}_{C}<\delta_5$
(later we will assume a stronger condition) and define
$t_1:=\sup\{t\in[0,r]\,|\;\n{u(s)-\bar u(s)}_C<
\delta_5\text{ for all }s\in[0,t]\}$. Note that
$0<t_1\le\min(\beta_{\w,x},r)$. Applying Lemma~\ref{4.lem3} we have
\[
 |\wit y(t)|\le k_0\,e^{-\alpha\,t}\,\n{x-\bar x}_C+3\,k_0\,\n{D_3F}_0\!\int_0^{t}
 e^{-\alpha(t-s)}\,\n{\wit y_s}_{C}\,ds \quad\text{for $t\in[0,t_1]$}\,,
\]
and hence
\[
 e^{\alpha t}\,|\wit y(t)|\le k_0\,
 \n{x-\bar x}_C+3\,k_0\,\n{D_3F}_0\!\int_0^{t}e^{\alpha s}\,\n{\wit y_s}_{C}\,ds
 \quad\text{for $t\in[0,t_1]$}\,.
\]
Let us define $r_1(t):=\sup\{e^{\alpha s}\,\n{\wit y_s}_{C}\,|\;0\le s\le t\}$.
It is not hard to check that
\[
 r_1(t)\le e^{\alpha r}k_0\,\n{x-\bar x}_C
 +3\,e^{\alpha r}\,k_0\,\n{D_3F}_0\!\int_0^{t}r_1(s)\,ds
 \quad\text{for $t\in[0,t_1]$}\,.
\]
Using the Gronwall Lemma, we obtain
\[
 r_1(t)\le e^{\alpha r}\,k_0\,e^{3\,e^{\alpha r}\,k_0\,\n{D_3F}_0\,r}\,\n{x-\bar x}_C
 \quad\text{for $t\in[0,t_1]$}\,.
\]
Consequently,
\[
 e^{\alpha t}\,|\wit y(t)|=e^{\alpha t}\,|\wit y_t(0)|\le e^{\alpha t}\,
 \n{\wit y_t}_{C}\le r_1(t)\le k_1^1\,\n{x-\bar x}_C\quad\text{for $t\in[0,t_1]$}\,,
\]
where $k_1^1:=e^{\alpha r}\,k_0\,e^{3\,e^{\alpha r}\,k_0\,\n{D_3F}_0\,r}>k_0\ge 1$,
and hence
\[
 |\wit y(t)|\le k_1^1\,e^{-\alpha\,t}\,\n{x-\bar x}_C\quad\text{for $t\in[0,t_1]$}\,.
\]
Now we assume that $\n{x-\bar x}_{C}\le\delta_5/k_1^1<  \delta_5$. Then
$|\wit y(t))|<\delta_5$ for any $t\in[-r,t_1]$, so that
$\n{u(t)-\bar u(t)}_C<\delta_5$  for all $t\in[0,t_1]$. An easy contradiction argument shows
that $t_1=r$ (so, in particular, $\beta_{\w,x}\ge r$) and hence that
\begin{equation}\label{4.desteor1}
 |y(t,\w, x)-y(t,\w,\bar x)|=|\wit y(t)|\le k_1^1\,e^{-\alpha\,t}\,\n{x-\bar x}_C
 \quad\text{for $t\in[0,r]$}\,.
\end{equation}
In particular,
\begin{equation}\label{4.desteor2}
 \n{\wit y_t}_C\le e^{\alpha t}\n{\wit y_t}_C\le
 k_1^1\,\n{x-\bar x}_{C}<\delta_5\quad\text{for $t\in[0,r]$}.
\end{equation}
\par
Let us consider now the case $t\geq r$. We assume that
$\n{x-\bar x}_{C}\le\delta_5/k_1^1$ (later the condition will be stronger)
and define $t_2:=\sup\{t\ge r\,|\;\n{u(s)-\bar u(s)}_C<\delta_5\text{ for all }s\in[0,t]\}$,
which satisfies $r<t_2\le\beta_{\w,x}$: see \eqref{4.desteor2}.
Applying Lemma~\ref{4.lem3}, now for $r\le t\le t_2$, and using \eqref{4.desteor2},
\[
\begin{split}
 |\wit y(t)|&\le k_0\,e^{-\alpha\,t}\,\n{x-\bar x}_{C}+
 3\,k_0\,\n{D_3F}_0\!\int_0^{r}e^{-\alpha(t-s)}\n{\wit y_s}_{C}\,ds\\
 &\quad+k_0\,\ep\int_{r}^{t}e^{-\alpha(t-s)}\n{\wit y_s}_{C}\,ds\\
 &\le (k_0+3\,k_0\,k_1^1\,r\,\n{D_3F}_0)\,e^{-\alpha\,t}\,\n{x-\bar x}_{C}
 +k_0\,\ep\,e^{-\alpha t}\!\int_{r}^{t}e^{\alpha s}\n{\wit y_s}_{C}\,ds\,.
\end{split}
\]
Let us call $k_1^2:=k_0+3\,k_0\,k_1^1\,r\,\n{D_3F}_0$.
We multiply the previous inequality by $e^{\beta\,t}$, so that, since $e^{(\beta-\alpha)t}<1$,
\begin{equation}\label{4.desteor6}
e^{\beta\,t}\,|\wit y(t)|\le
k_1^2\,\n{x-\bar x}_C+
k_0\,\ep\,e^{-(\alpha-\beta)t}\!\int_{r}^{t}e^{(\alpha-\beta)s}\,e^{\beta s}\,\n{\wit y_s}_{C}\,ds\,.
\end{equation}
Now we $r_2(t):=\sup\{e^{\beta s}\,\n{\wit y_s}_{C}\,|\;r\le s\le t\}$
for $t\in[r,t_2]$ and distinguish two cases.
\par
In the first case, we assume that
$r_2(t)=e^{\beta s^{*}}\n{\wit y_{s^{*}}}_{C}$ for $s^{*}\in[r, 2\,r]$,
and there exists $\theta^{*}\in[-r, 0]$ with $s^{*}\!+\theta^{*}\in[0, r]$
such that  $r_2(t)=e^{\beta s^{*}}|\wit y(s^{*}\!+\theta^{*})|$.
Since $s^{*}\!+\theta^{*}\in[0, r]$, we can apply~\eqref{4.desteor2} to conclude that
\[
 r_2(t)=e^{\beta s^*}|\wit y(s^{*}\!\!+\theta^{*})|
 \le e^{2 r \beta}\,k_1^1\,\n{x-\bar x}_C\,.
\]
Consequently,
\begin{equation}\label{4.desteor3}
 |\wit y(t)|\le e^{-\beta t}\,r_2(t)\le e^{2 r \beta}k_1^1\,e^{-\beta t}\,\n{x-\bar x}_C\,.
\end{equation}
\par
In the second case, which exhausts the possibilities,
$r_2(t)=e^{\beta s^{*}}\n{\wit y_{s^{*}}}_{C}$ for $s^{*}\in[r, t]$,
and there exists $\theta^{*}\in[-r, 0]$ with $s^{*}\!+\theta^{*}>r$ and
$r_2(t)=e^{\beta s^{*}}|\wit y(s^{*}+\theta^{*})|$.
We denote $\theta=s^{*}\!+\theta^{*}$, so that $s^*\le\theta+r$. Then, using \eqref{4.desteor6},
\[
\begin{split}
r_2(t)&=e^{\beta s^{*}}\,|\wit y(\theta)|\le e^{\beta r}e^{\beta\,\theta}\,|\wit y(\theta)|\\
&\le k_1^2\,e^{\beta r}\,\n{x-\bar x}_C
+k_0\,\ep\,e^{\beta r}\,e^{-(\alpha-\beta)\theta}\!\int_{r}^{\theta}
e^{(\alpha-\beta)s}e^{\beta s}\n{\wit y_s}_{C}\,ds\\
&\le k_1^2\,e^{\beta r}\,\n{x-\bar x}_C
+k_0\,\ep\,e^{\beta r}\,e^{-(\alpha-\beta)\theta}\,r_2(t)\!
\int_{r}^{\theta}e^{(\alpha-\beta)s}\,ds\\
&\le k_1^2\,e^{\beta r}\,\n{x-\bar x}_C
+\Frac{k_0\,\ep\,e^{\beta r}}{\alpha-\beta}\;r_2(t)\,.
\end{split}
\]
Thus, due to the choice of $\ep$,
\[
 \frac{1}{2}\;r_2(t)\le \left(1-\Frac{k_0\,\ep\,e^{\beta r}}{\alpha-\beta}\right)
 r_2(t)\le k_1^2\,e^{\beta r}\,\n{x-\bar x}_C.
\]
It follows easily that
\begin{equation}\label{4.desteor4}
 |\wit y(t)|\le e^{-\beta t}\,r_2(t)\le 2\,k_1^2\,e^{\beta r}e^{-\beta t}\,\n{x-\bar x}_C\,.
\end{equation}
Let us take $k_1:=\max\{e^{2 r \beta}\,k_1^1,\,2\,k_1^2\,e^{\beta r}\}>k_1^1$,
and take $\n{x-\bar x}_C\le\delta_5/k_1<\delta_5/k_1^1<\delta_5$. Then, using
\eqref{4.desteor1}, \eqref{4.desteor3} and~\eqref{4.desteor4}, we have
\begin{equation}\label{4.desteor5}
 |y(t,\w, x)-y(t,\w,\bar x)|=|\wit y(t)|\le k_1\,e^{-\beta\,t}\,\n{x-\bar x}_{C}<\delta_5
 \quad\text{for $t\in[0,t_2]$}\,.
\end{equation}
As before, an easy contradiction argument shows that $t_2=\infty$, and hence $\beta_{\w,x}=\infty$.
Let us define $\delta_1:=\delta_5/k_1$. The bound~\eqref{4.desyteor} follows
from this fact and \eqref{4.desteor5} for $t\ge 0$, and is trivial for
$t\in[-r,0]$ (since $k_1\ge 1)$.
\par
Finally, it is obvious that $u(t,\w,x)$ is defined for $t\in[0,\infty)$. The
bound~\eqref{4.desuteor} follows almost immediately from~\eqref{4.desyteor} and from
the definition of $\n{u(t)-\bar{u}(t)}_{C}$.
\smallskip\par
(2)$\Rightarrow$(3) We assume that (2) holds, take $(\w,\bar x)\in \mK$ and $x\ne\bar x$ with
$\n{x-\bar x}_C\le\delta_1$, and use again the notation~\eqref{4.notacion}.
\par
Let us fix $\ep>0$ small enough to apply Lemma \ref{4.lem3} (that is,
$\ep\le (4/3)\,\n{D_3F}_0$), and denote by
$\delta_5>0$ the constant that this lemma provides. We define
$\delta_2:=\min(\delta_1,\delta_5/k_1)$ (where $k_1$ is the constant appearing in (2))
and assume that
$\n{x-\bar x}_{\WW}\le\delta_2$. Then, according to~\eqref{4.desuteor},
$\n{u(t)-\bar u(t)}_C\le k_1e^{\beta r}\,\n{x-\bar x}_{C}
\le k_1e^{\beta r}\,\n{x-\bar x}_{\WW}\le\delta_5$ for $t\ge 0$,
and hence Lemma \ref{4.lem3} ensures that
\[
 |g(\wt,\bar{u}(t), u(t))|\le 3\,\n{D_3F}_0\,\n{u(t)-\bar{u}(t)}_{C} \quad\text{for $t\ge 0$}\,.
\]
\par
Recall now that~\eqref{4.ecteor} holds. We define $k_3:=
\max\{\n{L(\w,\bar x)}_{\Lin(C, \R^{n})}\,|\;(\w,\bar x)\in \mK\}$ and
use~\eqref{4.desuteor} to see that
\begin{equation}\label{4.desteor8}
\begin{split}
|\dot{y}(t,\w,x)-\dot{y}(t,\w,\bar x)|&\le (k_3+3\,\n{D_3F}_0)\,\n{u(t)-\bar{u}(t)}_{C}\\
&\le (k_3+3\,\n{D_3F}_0)\,k_1e^{\beta r}e^{-\beta\,t}\,\n{x-\bar x}_{C}\\
&\le (k_3+3\,\n{D_3F}_0)\,k_1e^{\beta r}e^{-\beta\,t}\,\n{x-\bar x}_{\WW}
\end{split}
\end{equation}
for every $t\geq0$. Now we define $k_4:=(k_3+3\,\n{D_3F}_0+1)\,k_1e^{\beta r}$ and combine
\eqref{4.desuteor},~\eqref{4.desteor8}, and the definition of $\n{u(t)-\bar{u}(t)}_{\WW}$
to conclude that
\[
 \n{u(t)-\bar{u}(t)}_{\WW}\le
 k_4\,e^{\beta r}e^{-\beta\,t}\,\n{x-\bar x}_{\WW}
 \quad\text{for $t\ge 0$}.
\]
Therefore, the assertion in (3) holds for $k_2:=k_4\,e^{\beta r}$.
\smallskip\par
(3)$\Rightarrow$(1) Let us take $(\w,\bar x)\in \mK$, $v\in\WW$ and $t\geq0$.
Theorem~\ref{3.teorlin} ensures that
\[
 \n{u_x(t,\w,\bar x)\,v}_{\WW}=\lim_{h\to0}\Frac{\n{u(t,\w,\bar x+h\,v)-
 u(t,\w,\bar x)}_{\WW}}{|h|}\:.
\]
Take $|h|$ small enough to guarantee $\n{h\,v}_{\WW}\le\delta_2$,
with $\delta_2$ provided by (3). Then, $\n{u(t,\w,\bar x+h\,v)-
u(t,\w,\bar x)}_{\WW}\le k_2\,e^{-\beta\,t}\,\n{h\,v}_{\WW}=
k_2\,e^{-\beta\,t}\,|h|\,\n{v}_{\WW}$.
Making again use of Theorem~\ref{3.teorlin},
$\n{w(t,\w,\bar x,v)}_{\WW}=
\n{u_x(t,\w,\bar x)\,v}_{\WW}\le k_2\,e^{-\beta\,t}\,\n{v}_{\WW}$
which, according to \eqref{2.lyap*} (for the linearized semiflow~\eqref{3.weak-skewlin}),
ensures that $\lambda_{\mK}<0$ and completes the proof of this implication.
\smallskip\par
In order to check that last assertion of the theorem it is enough to have a look to the choice
of $\beta$ in the proof of (1)$\Rightarrow$(2), and observe that the value of $\beta$
in (3) is the same one as in (2).\qed
\section{Weakening the hypotheses}\label{sec5}
Let $\Pi$ be the semiflow defined on $\W\times\WW$ by~\eqref{3.defPi}
from the family \eqref{3.eq} of FDEs.
In this section we work under the following assumptions, which are less restrictive than
those of the preceding one:
\begin{hipos}\label{5.hipos}
Conditions \hyperlink{3.H1}{H1} and \hyperlink{3.H21}{H2} hold,
and $\mK\subset\W\times\WW$ is
a positively $\Pi$-invariant compact set projecting over the whole base and
such that each one of its elements admits at least a backward extension in $\mK$.
\end{hipos}
As in the preceding sections, the set $\mK$ will be fixed throughout most of this one.
The first purpose now is to adapt to this less restrictive setting the characterization
of the exponential stability of $\mK$ in terms of its upper Lyapunov exponent.
The difference with respect to Theorem~\ref{4.teo} relies on the second equivalent
condition, which characterizes the exponential stability in terms of
$\n{\pu }_C$ instead of $\n{\pu }_{\WW}$. To formulate it,
we call
\begin{equation}\label{5.defro0}
 \rho_0:=\sup\{\n{\bar x}_{\WW}\,|\;(\w,\bar x)\in\mK\}\,.
\end{equation}
\begin{teor}\label{5.teo}
Suppose that Hypotheses~\ref{5.hipos} hold.
and let $\lambda_\mK$ and $\rho_0$ be respectively
given by Definition \ref{3.defly} and \eqref{5.defro0}.
The following statements are equivalent:
\begin{itemize}
  \item[(1)] $\lambda_{\mK}<0$.
  \item[(2)] There exists $\beta>0$ satisfying the following property: if we fix $\rho>\rho_0$,
  there exist constants $k_1=k_1(\rho)>0$ and $\delta_1=\delta_1(\rho)>0$ such that,
  if $(\w,\bar x)\in \mK$ and $(\w,x)\in\W\times\WW$ satisfy
  $\n{x}_{\WW}\le\rho$ and $\n{x-\bar x}_{C}\le\delta_1$,
  then the function $y(t,\w,x)$ is defined for $t\in[-r,\infty)$ and
      \[
      \quad|y(t,\w, x)-y(t,\w,\bar x)|\le k_1e^{-\beta t}\,\n{x-\bar x}_{C}\quad
      \text{for all $t\geq -r$}\,,
      \]
  so that
      \[
      \quad\n{u(t,\w, x)-u(t,\w,\bar x)}_C\le k_1e^{\beta r}e^{-\beta t}\,
      \n{x-\bar x}_{C}\quad\text{for  all $t\geq 0$}\,.
      \]
  \item[(3)] The set $\mK$ is exponentially stable; i.e., there exist $\beta>0$, $k_2\geq1$,
  and $\delta_2>0$ such that, if $(\w,\bar x)\in \mK$ and $(\w,x)\in\W\times\WW$
  satisfy $\n{x-\bar x}_{\WW}<\delta_2$, then the function $u(t,\w, x)$ is defined for
  $t\in[0,\infty)$, and
      \[
      \qquad \n{u(t,\w, x)-u(t,\w,\bar x)}_{\WW}\le k_2\,e^{-\beta t}\,
      \n{x-\bar x}_{\WW}\quad\text{for all $t\geq0$}\,.
      \]
  \end{itemize}
In addition, if\/ {\rm(1) holds}, we can take any $\beta\in(0,-\lambda_\mK)$ in {\rm(2)} and
{\rm(3)} (by changing the constants $\delta_1, k_1,\delta_2$ and $k_2$ if required).
\end{teor}
The proof of this theorem reproduces basically
that of Theorem~\ref{4.teo}. It is also based on three lemmas (Lemmas \ref{5.lem1},
\ref{5.lem2} and \ref{5.lem3}), whose statements are very similar to those of Section~\ref{sec4}
and whose proofs are almost identical. Just a little of previous work is required
in order to adapt everything to the less restrictive hypotheses we are considering now.
Given any $\gamma>0$, we denote
\[
 \mB_{\gamma}:=\{x\in\WW\,|\;\n{x}_{\WW}\le\gamma\}\,,
\]
which is a compact subset of $C$,
and represent by $L_1^{\gamma}$ and $L_2^{\gamma}$ the Lipschitz constants of
the functions $\tau$ and $D_2\tau$ on $\W\times \mB_{\gamma}$,
respectively provided by conditions \hyperlink{3.H21}{H2(3)} and \hyperlink{3.H21}{H2(2)}.
As in Section~\ref{sec4},
we take
\[
 r_0:=1+\sup\{\n{\bar x}_{C}\,|\;(\w,\bar x)\in \mK\}\,,
\]
and define
\[
\begin{split}
 |F|_0&:=\sup\{|F(\w, h, k)|\;|\;\w\in\W\,,\,|h|\le r_0,\,|k|\le r_0\}\,,\\
 \n{D_2F}_0&:=\sup\{\n{D_2F(\w, h, k)}_{\Lin(\R^{n}, \R^{n})}\,|\;
 \w\in\W\,,\,|h|\le r_0,\,|k|\le r_0\}\,,\\
 \n{D_3F}_0&:=\sup\{\n{D_3F(\w, h, k)}_{\Lin(\R^{n}, \R^{n})}\,|\;
 \w\in\W,\,|h|\le r_0,\,|k|\le r_0\}\,.
\end{split}
\]
Now we fix $\rho>\rho_0$ and define
\[
\begin{split}
 \rho^*&:=r_0+|F|_0+\rho_0+\rho\,,\\
 \n{D_2\tau}_0&:=\sup\{\n{D_2\tau(\w,\bar x+x)}_{\Lin(C, \R)}\,|
 (\w,\bar x)\in \mK,\,\n{x}_{\WW}\le \rho^*\}\,.
\end{split}
\]
To check that $\n{D_2\tau}_0<\infty$, we note that it
agrees with the supremum of $D_2\tau$ on a relatively compact
subset of $\W\times C$, which is finite by condition \hyperlink{3.H21}{H2(2)}.
We assume without restriction
that $|F|_0$, $\n{D_2F}_0$, $\n{D_3F}_0$, and
$\n{D_2\tau}_0$ are strictly positive.
\begin{lema}\label{5.lemaprevio}
Suppose that Hypotheses~\ref{5.hipos} hold, and fix $\rho>\rho_0$. We fix
$(\w,\bar x)\in \mK$ and $(\w,x)\in\W\times\mB_{\rho}$. Then,
\begin{align}
\n{\bar u(t)}_C\le r_0-1 &\quad\text{for $t\in[0,\infty)$}\,,\label{5.cotayt1}
\\[.15cm]
\n{\bar{u}(t)}_{\WW}\le\rho_0 &\quad \text{for $t\in[0,\infty)$}\,;\label{5.cotayt2}
\end{align}
and if
\begin{equation}\label{5.cota1}
\n{u(t)-\bar u(t)}_C\le 1\quad\text{\;\;for all $t\in[0,T]$}
\end{equation}
for a time $T\in(0,\beta_{\w,x})$, then
\begin{align}
\n{u(t)}_C\le r_0 &\quad\text{for $t\in[0,T]$}\,,\label{5.cotayt3}\\[.15cm]
\n{u(t)}_{\WW}\le r_0+|F|_0+\rho &\quad\text{for $t\in[0, T]$}\,,
\label{5.cotayt4}\\[0.15cm]
\n{u(t)-\bar{u}(t)}_{\WW}\le\rho^{*} &\quad\text{for $t\in[0, T]$}\,.\label{5.cotayt5}
\end{align}
The notation~\eqref{4.notacion} is used in this statement.
\end{lema}
\begin{proof}
Note that $(\wt,\bar u(t))\in\mK$ for all $t\ge 0$.
The inequalities \eqref{5.cotayt1} and~\eqref{5.cotayt2}
follow from this fact and the definitions of
$r_0$ and $\rho_0$. We assume~\eqref{5.cota1}, which together with \eqref{5.cotayt1}
ensures \eqref{5.cotayt3}. Before proving \eqref{5.cotayt4}, note that
\eqref{5.cotayt5} follows immediately from \eqref{5.cotayt2},
\eqref{5.cotayt4}, and the definition of $\rho^*$.
\par
In order to prove \eqref{5.cotayt4}, we take $t\in[0, T]$ and $s\in[-r, 0]$, and note that
\[
\dot{u}(t)(s)=\left\{
 \begin{array}{ll}
 \dot{y}(t+s) &\text{if $t+s\geq0$}\,,\\[0.15cm]
 \dot{x}(t+s) &\text{if $t+s\le0$}\,.
 \end{array}
 \right.
\]
If $t+s\le0$, then $|\dot{u}(t)(s)|=|\dot{x}(t+s)|\le\n{x}_{\WW}\le\rho$.
Assume now that $t+s\geq0$, so that $\dot{y}(t+s)=F(\w\pu(t+s),
y(t+s),y(t+s-\tau(\w\pu(t+s),y_{t+s})))$.
It follows from \eqref{5.cotayt3} that $|y(t+s)|\le r_0$ and
$|y(t+s-\tau(\w\pu(t+s), y_{t+s}))|\le r_0$. So, by definition of $|F|_0$,
we have $|\dot{y}(t+s)|\le |F|_0$. Hence
$\n{\dot{u}(t)}_{C}\le|F|_0+\rho$. Finally,
\eqref{5.cotayt3} yields
\[
 \n{u(t)}_{\WW}\le \n{u(t)}_{C}+\n{\dot{u}(t)}_{C}\le r_0+\n{\dot{u}(t)}_{C}\le
 r_0+|F|_0+\rho\,,
\]
as asserted.
\end{proof}
Now we give the statements of the lemmas which play, for the proof
of Theorem~\ref{5.teo}, the role played by Lemmas~\ref{4.lem1},
\ref{4.lem3} and \ref{4.lem2} in the proof of Theorem~\ref{4.teo}.
\begin{lema}\label{5.lem1}
Suppose that Hypotheses~\ref{5.hipos} hold, and fix $\rho>\rho_0$.
Then, for every $\ep>0$ there exists $\delta_3=\delta_3(\ep,\rho)\in(0,1]$ such that,
if $(\w,\bar x)\in \mK$, $(\w,x)\in\W\times\mB_{\rho}$, $T\in[0,\beta_{\w,x})$,
and $\n{u(t)-\bar{u}(t)}_{C}\le\delta_3$ for every $t\in[0, T]$, then
\[
\begin{split}
&\big|\bar{y}(t-\tau(\wt,u(t))) - \bar{y}(t-\tau(\wt,\bar{u}(t)))\\
&\qquad\qquad+\dot{\bar{y}}(t-\tau(\wt,\bar{u}(t)))\pu
D_2\tau(\wt,\bar{u}(t))(u(t) - \bar{u}(t))\big|\le\ep\,\n{u(t) - \bar{u}(t)}_{C}
\end{split}
\]
for every $t\in[0, T]$. The notation~\eqref{4.notacion} is used in this statement.
\end{lema}
\begin{lema}\label{5.lem2}
Suppose that Hypotheses~\ref{5.hipos} hold, and fix $\rho>\rho_0$.
Then, for every $\ep>0$ there exists $\delta_4=\delta_4(\ep,\rho)\in(0,1]$
such that, if $(\w,\bar x)\in \mK$, $(\w,x)\in\W\times\mB_{\rho}$, $T\in[0,\beta_{\w,x})$,
and $\n{u(t)-\bar{u}(t)}_{C}\le\delta_4$ for every
$t\in[0, T]$, then
\[
 |\wit y(t-\tau(\wt,u(t)))-\wit y(t-\tau(\wt,\bar{u}(t)))|\le\left\{\!
 \begin{array}{ll}
 2\,\n{u(t)-\bar{u}(t)}_{C}&\text{if}\;\;\;t\in[0,r]\,,\\[0.1cm]
 \ep\,\n{u(t)-\bar{u}(t)}_{C}&\text{if}\;\;\;t\in[r,T]\,.
 \end{array}
 \right.
\]
The notation~\eqref{4.notacion} is used in this statement.
\end{lema}
\begin{lema}\label{5.lem3}
Suppose that Hypotheses~\ref{5.hipos} hold, fix $\rho>\rho_0$, and define
$g\colon\mK\times\WW\to\R^{n}$ by
\[
 g(\w,\bar x, x):=F(\w,x(0), x(-\tau(\w,x)))-F(\w,\bar x(0),
 \bar x(-\tau(\w,\bar x)))-L(\w,\bar x)(x-\bar x)\,,
\]
where $L(\w,\bar x)$ is given by~\eqref{3.defL}.
Then, for every $\ep\in(0,(4/3)\n{D_3F}_0]$ there exists
$\delta_5=\delta_5(\ep,\rho)\in(0,1]$ such that,
if $(\w,\bar x)\in \mK$, $(\w,x)\in\W\times\mB_{\rho}$, $T\in[0,\beta_{\w,x})$,
and $\n{u(t)-\bar{u}(t)}_{C}\le\delta_5$ for every $t\in[0, T]$, then
\[
|g(\wt,\bar{u}(t), u(t))|\le\left\{
 \begin{array}{ll}
 3\,\n{D_3F}_0\,\n{u(t)-\bar{u}(t)}_{C}&\text{if}\;\;\;t\in[0,r]\,,\\[0.1cm]
 \ep\,\n{u(t)-\bar{u}(t)}_{C}&\text{if}\;\;\;t\in[r,T]\,.
 \end{array}
 \right.
\]
The notation~\eqref{4.notacion} is used in this statement.
\end{lema}
This completes the summary of ideas regarding the proof of Theorem~\ref{5.teo}.
\par
The following consequence of Theorem~\ref{5.teo} in the case of minimal base flow
will play a fundamental role in the rest of the paper.
Recall that the set $\mK$ is a {\em $k$-cover of $(\W,\sigma,\R)$} if each fiber
$\mK_\w:=\{x\in\WW\,|\;(\w,x)\in\mK\}$ contains exactly $k$ elements.
\begin{coro}\label{5.coro}
Suppose that the base flow $(\W,\sigma,\R)$ is minimal,
that Hypotheses~\ref{5.hipos} hold, and that $\lambda_{\mK}<0$. Then,
there exists $k\in\N$ such that $\mK$ is a $k$-cover of $(\W,\sigma,\R)$,
and the semiflow $(\mK,\Pi,\R^+)$ admits a flow extension.
In addition,
\begin{itemize}
\item[\rm(i)]
for each $\wit\w\in\W$ there exist a neighborhood $\mU_{\,\wit\w}\subset\W$ of
$\wit\w$ and
$k$ continuous maps $x_1,\ldots,x_k\colon\mU_{\,\wit\w}\to\WW$ such that
\begin{equation}\label{5.fibra}
 \mK_{\w}=\{\bar x\in\WW\,|\;(\w,\bar x)\in\mK\}=\{x_1(\w),\ldots,x_k(\w)\}
\end{equation}
for all $\w\in\mU_{\,\wit\w}$.
\item[\rm(ii)] The set $\mK$ is the disjoint union of a finite number of
minimal sets $\mM_1,\ldots,$ $\mM_l$, where $\mM_j$ is an exponentially stable
$m_j$-cover of the base for $j=1,\ldots,l$.
\end{itemize}
\end{coro}
\begin{proof}
Theorem \ref{5.teo} ensures that $\mK$ is exponentially stable, so that it is
uniformly asymptotically stable. Theorem 3.5 of Novo {\em et al.} \cite{noos5},
which is based on previous results of Sacker and Sell \cite{SackerSell},
proves that $\mK$ is a $k$-cover of the base for a $k\in\N$. The fact that
$(\mK,\Pi,\R^+)$ admits a flow extension follows for instance from Theorem 3.4 of
\cite{noos5}.
\smallskip\par
(i) This assertion can be easily proved by combining
two facts: first, the closed character of $\mK$ ensures the continuity of the map
$\w\mapsto\mK_\w$ in the Hausdorff topology of the set of compact subsets
of $\WW$ (see Theorem 3.3 of \cite{noos5});
and second, $\mK_\w$ always contains $k$ elements.
\smallskip\par
(ii) Let $\mM\subseteq\mK$ be a minimal set. It is obvious that $\lambda_\mM<0$, and hence,
as we have already proved, $\mM$ is an exponentially stable $m$-cover of the
base with $m\le k$. It is easy to deduce from
the existence of flow extensions on $\mK$ and $\mM$ that $\mK-\mM$ is also positively
$\Pi$-invariant. Let us now take a sequence $(\w_k)$ with limit $\w$. Theorem
 of \cite{noos5} ensures that $(\mK_{\w_k})$ and $(\mM_{\w_k})$
 respectively converge to $\mK_\w$ and $\mM_\w$
in the Hausdorff topology of the set of compact subsets of $\WW$, and
it is not hard to deduce from here that $((\mK-\mM)_{\w_k})$ converges to
$(\mK-\mM)_\w$, and hence that $\mK-\mM$ is a positively $\Pi$-invariant
compact set. Obviously, $\lambda_{\mK-\mM}<0$. Altogether,
we see that the set $\mK-\mM$ satisfies the same conditions as $\mK$, so that it is a
$(k-m)$-cover of the base. Repeating the process a finite number of times (at most
$k-1$) leads us to the desired conclusion.
\end{proof}
This section contains two more results, both of them referred to the case in which the base flow
$(\W,\sigma,\R)$ is minimal. The last one, Theorem~\ref{5.teo2}, extends the information
given by Theorem~\ref{5.teo}: it proves that, if each minimal subset of a
positively $\Pi$-invariant compact set $\mP$ has negative upper-Lyapunov index,
then $\mP$ contains a finite number $l$ of minimal sets, and its
connected components are the positively $\Pi$-invariant subsets determined by the
domains of attraction of its minimal subsets.
Recall that the {\em domain of attraction of a minimal set $\mM$ with $\lambda_\mM<0$} is
defined, in this skew-product setting, by
\[
\begin{split}
 \mD(\mM):=\Big\{(\w,x)\in\W\times\WW\,|&\;\text{there exists }
 (\w,\bar x)\in\mM\\
 &\text{ with }\lim_{t\to\infty}\n{u(t,\w,x)-\bar u(t,\w,\bar x)}_{\WW}=0\Big\}\,.
\end{split}
\]
Note that we are not assuming the existence of backward extensions for the
elements of $\mP$. The proof of Theorem~\ref{5.teo} relies on Proposition~\ref{5.prop}, which
shows several properties for $\mD(\mM)$ in the case that $\lambda_\mM<0$.
\begin{prop}\label{5.prop}
Suppose that the base flow $(\W,\sigma,\R)$ is minimal, that
conditions {\rm \hyperlink{3.H1}{H1}} and {\rm \hyperlink{3.H21}{H2}} hold,
and that $\mM\subset\mU\subseteq\W\times\WW$
is a minimal set with $\lambda_{\mM}<0$. Then,
\begin{itemize}
\item[\rm(i)] the domain of attraction of $\mM$, $\mD(\mM)$,
is an open and connected positively $\Pi$-invariant set.
\item[\rm(ii)]
For all $\beta\in(0,-\lambda_\mM)$ and all compact set $\mP\subset\mD(\mM)$
there exists a constant $k=k(\beta,\mP)$ such that, for every $(\w,x)\in\mP$, there exists
$(\w,\bar x)\in\mM$ with
\[
 \n{u(t,\w,x)-u(t,\w,\bar x)}_{\WW}\le k\,e^{-\beta\,t}
 \quad\text{for $t\ge 0$}\,.
\]
\end{itemize}
\end{prop}
\begin{proof}
(i) It is obvious that the set $\mD(\mM)$ is positively $\Pi$-invariant.
We fix $\beta\in(0,-\lambda_\mM)$ and apply Theorem
\ref{5.teo} to find $k_2\ge 1$ and $\delta_2>0$ such that, if $(\w,\bar x)\in \mM$ and
$(\w,x)\in\W\times\WW$ satisfy $\n{x-\bar x}_{\WW}<\delta_2$, then
$\beta_{\w,x}=\infty$ and
$\n{u(t,\w, x)-u(t,\w,\bar x)}_{\WW}\le k_2\,
e^{-\beta t}\,\n{x-\bar x}_{\WW}$ for all $t\ge 0$.
Let us fix $(\wit\w,\wit x)\in\mD(\mM)$ and look for $(\wit\w,\bar x)\in\mM$ and
$t_0>0$ such that
$\n{u(t_0,\wit\w,\wit x)-u(t_0,\wit \w,\bar x)}_{\WW}\le\delta_2/3$.
We also look for $\delta_0>0$ such that, if $d_\W(\w,\wit \w)<\delta_0$ and
$\n{x-\wit x}_{\WW}<\delta_0$, then $\n{u(t_0,\w,x)-
u(t,\wit\w,\wit x)}_{\WW}\le \delta_2/3$.
And we finally look for $\delta^0\le\delta_0$ such that, if
$d_\W(\w,\wit\w)<\delta^0$, then $\w\in\mU_{\,\wit\w}$ and in addition,
if $\bar x=x_i(\wit\w)$, then $\n{x_i(\w)-x_i(\wit\w)}_{\WW}
<\delta_0$ (see Corollary~\ref{5.coro}).
Finally, we take $(\w,x)$ with $d_\W(\w,\wit\w)<\delta^0$ and
$\n{x-\wit x}_{\WW}<\delta_0$. Then we have
\[
\begin{split}
 &\n{u(t_0,\w,x)-u(t_0,\w,x_i(\w))}_{\WW}\le
  \n{u(t_0,\w,x)-u(t_0,\wit\w,\wit x)}_{\WW}\\
 &+\n{u(t_0,\wit\w,\wit x)-u(t_0,\wit\w,x_i(\wit\w))}_{\WW}\!+
 \n{u(t_0,\wit\w,x_i(\wit\w))-u(t_0,\w,x_i(\w))}_{\WW}\!\le\delta_2\,.
\end{split}
\]
Therefore,
\[
\begin{split}
 &\n{u(t,\w,x)-u(t,\w,x_i(\w))}_{\WW}\\
 &\qquad\qquad=\n{u(t-t_0,\w\pu t_0,u(t_0,\w,x))-
 u(t-t_0,\w\pu t_0,u(t_0,\w,x_i(\w)))}_{\WW}\\
 &\qquad\qquad\le \delta_2\,k_2\,e^{-\beta(t-t_0)}
\end{split}
\]
for all $t\ge t_0$.
This inequality ensures that $(\w,x)\in\mD(\mM)$, and hence that $\mD(\mM)$ is open
in $\W\times\WW$, as asserted.
\par
In order to prove that $\mD(\mM)$ is connected, write $\mD(\mM)\subseteq\mV_1\cup\mV_2$
for two disjoint open subsets $\mV_1$ and $\mV_2$ of $\W\times\WW$. Since $\mM$ is
connected (as
any minimal set),
then it is contained in one of these sets, say $\mM\subset\mV_1$. But any point
$(\w,x)\in\mD(\mM)\cap\mV_2$ is connected with $\mV_1$ by a positive semiorbit, which together
with the positively $\Pi$-invariance of $\mD(\mM)$ shows that $\mD(\mM)\cap\mV_2$ is empty.
The conclusion is that $\mD(\mM)$ is connected, which completes the proof of (i).
\smallskip\par
(ii) We fix again $\beta\in(0,-\lambda_\mM)$ and take constants $k_2\ge 1$ and $\delta_2>0$
with the same properties as in the proof of (ii). Let us define
\[
 \mD_0(\mM):=\Big\{(\w,x)\in\W\times\WW\,|\;\text{there exists }
 (\w,\bar x)\in\mM\text{ with }\n{x-\bar x}_{\WW}<\delta_2\Big\}.
\]
It follows easily from Corollary~\ref{5.coro} that
$\mD_0(\mM)$ is an open subset of $\mD(\mM)$.
Note also that there exists
$t_0>0$ such that $\Pi(t,\mD_0(\mM))\subseteq \mD_0(\mM)$ for all $t\ge t_0$,
as easily deduced from Theorem~\ref{5.teo} and the definition of $\delta_2$.
\par
Let us take a compact set $\mP\subset\mD(\mM)$. The next goal is
to check that there exists
$t_1=t_1(\mP)>0$ such that $\Pi(t_1,\mP)\subset\mD_0(\mM)$.
The definition of $\mD(\mM)$ ensures that, for any $(\wit\w,\wit x)\in\mP$, there exists
$t_{\wit\w,\wit x}>0$ such that
$(\wit\w\pu t_{\wit\w,\wit x},u(t_{\wit\w,\wit x},\wit\w,\wit x))\in\mD_0(\mM)$ and,
since $\mD_0(\mM)$ is open,
the same happens for all the points $(\w,x)$ in a
neighborhood $\mV_{\wit\w,\wit x}\subset\W\times\WW$ of $(\wit\w,\wit x)$.
Hence, $(\wt,u(t,\w,x))\in\mD_0(\mM)$ for all $t\ge t_0+t_{\wit\w,\wit x}$
and all $(\w,x)\in\mV_{\wit\w,\wit x}$. The compactness of $\mP$ proves the existence of~$t_1$.
\par
Therefore, if $(\w,x)\in\mP$, then there exists $(\w\pu t_1,\bar y)\in\mM$ such that
$\n{u(t_1,\w,x)-\bar y}_{\WW}\le$ $\delta_2$. Since $\mM$ admits a flow extension,
there exists $\bar x=u(-t_1,\w\pu t_1,\bar y)$.~So,
\[
\begin{split}
 &\n{u(t,\w,x)-u(t,\w,\bar x)}_{\WW}\\
 &\quad\;=
 \n{u(t-t_1,\w\pu t_1,u(t_1,\w,x))-u(t-t_1,\w\pu t_1,u(t_0,\w,\bar x))}_{\WW}
 \le k\,\delta_2\,e^{-\beta\,(t-t_1)}
\end{split}
\]
for $t\ge t_1$. The assertion in (ii) follows easily from the uniform continuity of
$\Pi\colon[0,t_0]\times\mM\to C$ and $\Pi\colon[0,t_0]\times\mP\to C$
(ensured by Theorem~\ref{3.flujocontinuo}(v))
and the boundedness of $F$ on the compact subsets of $\W\times\R^{2n}$.
\end{proof}
\begin{teor}\label{5.teo2}
Suppose that the base flow $(\W,\sigma,\R)$ is minimal and
that conditions {\rm \hyperlink{3.H1}{H1}} and {\rm \hyperlink{3.H21}{H2}} hold.
Let $\mP\subset\W\times\WW$ be a positively $\Pi$-invariant
compact set such that, for any minimal subset $\mM\subseteq\mP$, it is
$\lambda_\mM<0$. Then,
\begin{itemize}
\item[\rm(i)] The omega-limit set $\mO(\w,x)$ of any $(\w,x)\in\mP$ is a minimal subset of $\mP$.
\item[\rm(ii)] $\mP$ contains a finite number of minimal sets, $\mM_1,\ldots,\mM_l$.
\item[\rm(iii)] If $\mD(\mM_1),\ldots,\mD(\mM_l)$ are the corresponding domains of attraction,
then the sets $\mP\cap\mD(\mM_j)$ are compact and connected
positively $\Pi$-invariant sets for $j=1,\ldots,l$, they are pairwise disjoint,
and $\mP=\bigcup_{j=1}^{\,l}\big(\mP\cap\mD(\mM_j)\big)$.
\end{itemize}
In particular, if $\mP$ is connected, it contains just one minimal set.
\end{teor}
\begin{proof}
(i) Let us take a point $(\w,x)\in\mP$, a minimal set $\mM\subseteq\mO(\w,x)\subseteq\mP$,
and a point $(\w,y)\in\mM$. By hypothesis, $\lambda_\mM<0$.
We take a sequence $(t_k)\uparrow\infty$ such that
$(\w,y)=\lim_{k\to\infty}(\w\pu t_k,u(t_k,\w,x))$. This fact, the
open character of $\mD(\mM)$ established in Proposition~\ref{5.prop}(i),
and the existence of backward orbits in $\mM$, ensure the existence of a point
$(\w,\bar x)\in\mM$ such that $\lim_{t\to\infty}\n{u(t,\w,\bar x)-u(t,\w,x)}_{\WW}=0$.
This means that $\mO(\w,x)=\mO(\w,\bar x)=\mM$, which proves (i).
%
\smallskip\par
(ii) Suppose that $\mP$ contains two different minimal subsets $\mM_1$ and $\mM_2$.
It is obvious that $\mD(\mM_1)$ and $\mD(\mM_2)$ are disjoint. On the other hand,
it follows from (i) that $\mP$ is contained in the union of the domains
of attraction of all its minimal subsets, each one of which is an open set
with nonempty intersection with $\mP$. Hence, (ii) follows from
compactness of $\mP$.
\smallskip\par
(iii) It follows from (i) and (ii)
that $\mP=\bigcup_{j=1}^{\,l}\big(\mP\cap\mD(\mM_j)\big)$. Our first goal is to prove
that the positively $\Pi$-invariant set $\mP\cap\mD(\mM_j)$ is closed (and hence
compact) for $j=1,\ldots,l$. Let us fix $j\in\{1,\ldots,l\}$
and take a sequence $((\w_m,x_m))\in\mP\cap\mD(\mM_j)$
with limit $(\wit\w,\wit x)\in\mP$. We look for $k\in\{1,\ldots,l\}$ such that
$(\wit\w,\wit x)\in\mP\cap\mD(\mM_k)$. Since this set is open in $\mP$, there
exists $m_0$ such that $(\w_{m_0},x_{m_0})\in\mP\cap\mD(\mM_k)$, and since
$\mD(\mM_j)\cap\mD(\mM_k)$ is empty if $k\ne j$, then $k=j$. That is,
$\mP\cap\mD(\mM_j)$ is closed, as asserted.
\par
In order to prove that each set
$\mP\cap\mD(\mM_j)$ is connected, we assume by contradiction that for an index
$j\in\{1,\ldots,l\}$ we can write $\mP\cap\mD(\mM_j)\subset\mV_1\cup\mV_2$
for two disjoint open subsets $\mV_1$ and $\mV_2$ of $\W\times\WW$.
Since $\mM_j$ is a connected subset of $\mP\cap\mD(\mM_j)$, we can assume without
restriction that $\mM_j\subset\mV_1$. And since $\mP\cap\mD(\mM_j)$ is a
positively $\Pi$-invariant subset of $\mD(\mM_j)$, we conclude that
$\mP\cap\mD(\mM_j)\cap\mV_2$ is empty. This completes the proof of (iii),
\smallskip\par
The last assertion of the theorem follows trivially from (iii) together with the
open character of $\mD(\mM_j)$ ensured by Proposition~\ref{5.prop}(i).
\end{proof}
\begin{nota}
It is important to emphasize the fact that, if the flow $(\W,\sigma,\R)$
is almost periodic and $\mM$ is a minimal $m$-cover of the base admitting a flow extension,
then the flow $(\mM,\Pi,\R)$ is also almost periodic: see \cite{SackerSell}, Theorem 6
of Chapter 3. Therefore, Corollary \ref{5.coro} ensures the following property.
Assume that our family~\eqref{3.eq} is constructed by the
usual hull procedure (summarized in the
Introduction) from a single FDE $\dot{y}(t)=f(t, y(t), y(t-\wit\tau(t,y_{t})))$
given by a uniformly almost periodic pair $(f,\wit\tau)$.
Then the existence of a minimal set
$\mM\subset\W\times\WW$
with $\lambda_{\mM}<0$ ensures the existence of exponentially stable
almost-periodic solutions of the initial system. Note also that the existence of
such a set $\mM$ is ensured by the existence of a bounded and uniformly exponentially
stable solution on $[0,\infty)$ of the initial system:
it is easy to check that the omega-limit set of
such a solution for the flow $\Pi$ associated to the
family of FDE defined on the corresponding hull $\W$
is a minimal subset of $\W\times\WW$
(just repeat the proof of Theorem~\ref{5.teo2}(i)), which in addition
is exponentially stable.
\end{nota}

\end{document}